\documentclass[12pt,reqno]{amsart}
\usepackage{amsfonts}
\usepackage{amsmath,amssymb,amsthm, mathrsfs,amsopn}%
\usepackage{a4wide}
\allowdisplaybreaks
\numberwithin{equation}{section}
 \usepackage[pagewise]{lineno}\linenumbers\nolinenumbers
\usepackage{graphicx}
\usepackage{epstopdf}
\usepackage{float}
\usepackage{stfloats}
\usepackage{subfigure}
\usepackage{soul}

\usepackage{url}
\usepackage{color}
\usepackage[colorlinks, linkcolor=blue, citecolor=blue]{hyperref}
\newtheorem{theorem}{Theorem}[section]
\newtheorem{proposition}[theorem]{Proposition}

\newtheorem{lemma}[theorem]{Lemma}

\theoremstyle{definition}

\newtheorem{definition}{Definition}[section]

\newtheorem{remark}{Remark}[section]

\newcommand{\R}{\mathbb{R}}

\newcommand{\eeq}{\end{equation}}
\newcommand{\beq}{\begin{equation*}}
\newcommand{\lab}{\label}

\begin{document}
	
	\title[  prescribed the mass solution]{   Existence and local uniqueness of multi-spike solutions for Br\'{e}zis-Nirenberg problem with prescribed mass }

	\author{Zongyan Lv}
	\address{ Zongyan Lv,~ Center for Mathematical Sciences, Wuhan University of Technology, Wuhan 430070, P. R. China	}
	\email{zongyanlv0535@163.com}
	\author{ Xiaoyu Zeng}
	\address{ Xiaoyu Zeng,~Center for Mathematical Sciences, Wuhan University of Technology, Wuhan 430070, P. R. China}
	\email{ xyzeng@whut.edu.cn}
	\author{Huan-Song Zhou}
	\address{ Huan-Song Zhou,~Center for Mathematical Sciences, Wuhan University of Technology, Wuhan 430070, P. R. China}
	\email{hszhou@whut.edu.cn}
	\date{}
	
	\begin{abstract}
	In this paper, 	we consider  the following  Br\'{e}zis-Nirenberg problem with prescribed $ L^2$-norm (mass) constraint: 
		\begin{equation*}
			\begin{cases}
				-\Delta u=|u|^{2^*-2} u +\lambda_\rho u\quad \text { in } \Omega, \\
				u>0, \quad 	u \in H_0^1(\Omega), \quad \int_{\Omega} u^2dx=\rho,
			\end{cases}
		\end{equation*}
		where   $N \geqslant 6$,  $2^*=2 N /(N-2)$  is the critical Sobolev exponent, $\rho>0$ is a given small constant and $\lambda_\rho>0$ acts as an Euler-Lagrange multiplier.    For any $k\in \mathbb{R}^+$, we  construct a $k$-spike solutions in some suitable bounded domain $\Omega$. Our results extend those in \cite{BHG3,DGY,SZ}, where the authors obtained  one or two positive solutions corresponding to the (local) minimizer or mountain pass type critical point for the energy functional of above equation.   Furthermore, using  blow-up analysis and local Pohozaev identities arguments, we prove that the $k$-spike solutions are locally unique. Compared to the standard Br\'{e}zis-Nirenberg problem without the mass constraint, an additional difficulty arises in estimating the error caused by the differences in the Euler-Lagrange multipliers corresponding to different solutions. We overcome this difficulty by introducing novel observations and estimates related to the kernel of the linearized operators.
        
        	


	\end{abstract}
	
	\maketitle
	
	\section{introduction}

The aim of this paper is to study the existence,  multiplicity and local uniquness of the  solutions for the following $L^2$-constrained   Br\'{e}zis-Nirenberg problem
			\begin{equation}\lab{equ:Main-Problem}
		\begin{cases}
			-\Delta u=|u|^{2^*-2} u +\lambda_\rho u\quad \text { in } \Omega \subset \mathbb{R}^N, \\
			u \in H_0^1(\Omega), \quad \int_{\Omega} u^2dx=\rho,
		\end{cases}
	\end{equation}
	where $\Omega$ is a bounded   domain in $\mathbb{R}^N$ with $N \geq 6$, $\lambda_\rho  \in\left(0, \lambda_1(\Omega)\right)$ is a suitable Euler-Lagrange multiplier   with  $\lambda_1(\Omega)$ being  the first eigenvalue of $-\Delta $ with zero Dirichlet
	boundary condition on $\Omega$.
	
	One primary motivation for studying \eqref{equ:Main-Problem} relies on the investigation of the  standing waves solutions for the nonlinear Schr\"{o}dinger equation
	\begin{align}\label{equ:241111-e2}
			-i \frac{\partial \Phi}{\partial t}=\Delta \Phi+f(|\Phi|) \Phi \quad \text { in } \mathbb{R}^+ \times \mathbb{R}^N,
	\end{align}
which emerges in the theoretical  invsetigation  on  Bose-Einstein condensation (\cite{BEC})  or  nonlinear optic problem (\cite{GPA,WK,LSJJ,LSJ}). Especially, the nonlinear Schr\"{o}dinger equation on  bounded domains in nonlinear optics  could describe the propagation of laser beams in hollow-core fibers (\cite{GF}).

The ansatz $\Phi(x, t)=e^{i \lambda t} u(x)$ for standing wave solutions induces the elliptic equation
	\begin{align}\label{equ:241111-e1}
-\Delta u+\lambda u=f(|u|)u \quad \text { in } \mathbb{R}^N.
	\end{align}
We generally  call  \eqref{equ:241111-e1} a fixed frequency problem provided that the parameter $\lambda \in \mathbb{R}$ is fixed. In the past decades, a lot of work, such as the existence and qualitative properties, have been done on this type  problem, see e.g.\cite{AS,BW,BN,CSS,MS,ZZ}. 

The $L^2$-norm of solution   $\Phi(t, \cdot)$ to \eqref{equ:241111-e2}, which represents the {\em total mass} of the system,  remains  invariant  for all  $t\in \mathbb{R}^+$.  As a consequence, the quantity  $\|\Phi(t, \cdot)\|_{L^2}$  serves as a crucial parameter  determining  the long time behavior of $\Phi(t, \cdot)$, as well as the threshold for finite time blow-up or orbital stabilities, see  for instance \cite{Caz,CL}.  From this perspective, investigating the $L^2$-normed  solution, namely,   solution \eqref{equ:241111-e1} satisfying
\begin{align}\label{constraint}
	\int_{\mathbb{R}^N} u^2dx=a,
\end{align}
 is always an important issue. Correspondingly,  the frequency $\lambda_a $ must be treated  as an  unknown Euler-Lagrange multipler  in \eqref{equ:241111-e1}. 
In general, the system  \eqref{equ:241111-e1}-\eqref{constraint}  is called fixed mass problem and the solution pair $(u,\lambda_a)$ is called a normalized solution.

In recent years, normalized solutions to the nonlinear Schr\"odinger equations have attracted increasing attention from scholars. A classical approach to finding normalized solutions of \eqref{equ:241111-e1}-\eqref{constraint} is through the variational method: one seeks critical points of the functional $I(u)$ constrained on $\mathcal{M}_a$, where
\begin{align*}
	I(u):=\frac{1}{2}\int_{\mathbb{R}^N}|\nabla u|^2dx-\int_{\mathbb{R}^N}F(u) dx, \quad \text{where} 
	\quad 	F(s):=\int_0^s f(t) d t,
\end{align*}
and
\begin{align*}
\mathcal{M}_a:=\big\{u\in H^1(\R^N): \int_{\mathbb{R}^N} u^2dx=a\big\}.
\end{align*}
Here, $\lambda$ appears as a Euler-Lagrange multiplier in \eqref{equ:241111-e1}.

When working on the whole space $\mathbb{R}^N$, the literature in this direction is extensive, and it is beyond the scope of this work to provide a comprehensive summary. Here, we only highlight some key contributions involving mass-supercritical nonlinearities, such as $f(|u|)u=|u|^{p-2}u$, with$2+\frac{4}{N}<p<2^*$.   
 To the best of our knowledge, the first result addressing the mass-supercritical case was established in \cite{JJ}, where Jeanjean proved the existence of normalized solutions by developing a mountain pass theorem on the
 $L^2$ constrained manifold. This work overcame the challenge that the energy functional is unbounded from below under the  $L^2$-constraint.  In \cite{TBNS}, Bartsch and Soave proved that the Pohozaev manifold is a natural constraint and introduced a standard method based on this manifold. Subsequently, the results of \cite{TBNS,JJ} were extended by Jeanjean and Lu \cite{LJLS} and Bieganowski and Mederski \cite{BJ} under some weaker assumptions. We also note that Soave \cite{SN,SN1} made numerous contributions to the study of combined power nonlinearities, i.e., $$f(|u|) u=|u|^{q-2} u+\beta|u|^{p-2} u, 2<p<2+\frac{4}{N}<q \leq 2^*.$$
Additionally, we refer to \cite{LJTL,WW} for treatments of nonlinearities involving the Sobolev critical exponent. Further works  addressing  systems of Schr\"odinger equations can be find in \cite{TLN,TBNS-1,TBNS,BZZ}.

 When working on bounded domains, the existence of normalized solutions has been studied in several works, including \cite{TQZ,CDCE,BHSG,BHG1,BHG2,BHG3,DG,DGY,SZ,WJ1,WJ2}. It is important to note that the methods developed for $\mathbb{R}^N$ cannot be directly applied to bounded domains, as the lack of dilation invariance prevents the use of tools like the Pohozaev manifold. This limitation significantly complicates the analysis.  Furthermore, if we consider the Sobolev critical case, the loss of compactness introduces additional challenges to the problem.  To the best of our knowledge,  there are only these papers \cite{BHG3,DGY,SZ} which
 have addressed the existence of normalized solutions to the  Br\'{e}zis-Nirenberg problem. In these works, the authors employed variational arguments to establish the existence of (local) minimizers or mountain pass-type solutions for the energy functional associated with \eqref{equ:Main-Problem}.

 In this paper, we aim to study the multiplicity, quantitative properties, and local uniqueness of solutions to \eqref{equ:Main-Problem}. Specifically, for any fixed $k\in\mathbb{N}^+$, we will prove  that \eqref{equ:Main-Problem} admits at least 
$k$ distinct solutions with $k$ spikes in suitably chosen bounded domains. Moreover, we prove that these $k$-spike solutions are {\em locally unique}, meaning that the number of such solutions can be precisely determined for domains with specific topological properties.
For these purposes, we introduce some notations. We first recall 
from \cite{G} that the equation $-\Delta u=u^{\frac{N+2}{N-2}}$ in $\mathbb{R}^N$ has a family of solutions
\begin{align*}
U_{x, \lambda}(y)=\frac{(N(N-2))^{(N-2) / 4}\lambda^{(N-2) / 2}}{\left(1+\lambda^2|y-x|^2\right)^{(N-2) / 2}} \text{ with }x \in \mathbb{R}^N, \lambda \in \mathbb{R}^{+},
\end{align*}
and  all $U_{x, \lambda}(y)$ potimize  the following  best Sobolev inequality
	\begin{align*}
	\mathcal{S}:=\inf\big\{ \|\nabla u\|^2_{L^2(\mathbb{R}^N)}\big/\| u\|_{L^{2^*}(\mathbb{R}^N)}: 0\neq u\in H^1(\mathbb{R}^N)\big\}.
\end{align*}
Now let $\mathrm{H}_0^1\left(\Omega\right)$ be the Hilbert space equipped with the usual inner product
$\langle u, v\rangle:=\int_{\Omega} \nabla u \nabla v$, which induces the norm $\|u\|=\left(\int_{\Omega}|\nabla u|^2\right)^{1 / 2}$.
For any $x \in \Omega$ and $\mu \in \mathbb{R}^{+}$, we define
$$
E_{x, \mu}=\Big\{v \in H_0^1(\Omega) \Big\lvert\,\Big\langle\frac{\partial P U_{x, \mu}}{\partial \mu}, v\Big\rangle=\Big\langle\frac{\partial P U_{x, \mu}}{\partial x_i}, v\Big\rangle=0, \text { for } i=1, \cdots, N\Big\}.
$$
Moreover, for any given $f \in H^1(\Omega)$, let $P$ denote the projection from $H^1(\Omega)$ onto $H_0^1(\Omega)$, i.e., 

$$
\text{$u:=P f$ is the solution of }\begin{cases}\Delta u=\Delta f, & \text { in } \Omega, \\ u=0, & \text { on } \partial \Omega .\end{cases}
$$


Let
 \begin{align*}
 	a^k:=\left(a_1, \cdots, a_k\right) \quad  \text{and} \quad  \mu^k:=\left(\mu_1, \cdots, \mu_k\right),
 \end{align*}
and $\Psi_k: \Omega^k \times\left(\mathbb{R}^{+}\right)^k \rightarrow \mathbb{R}$ be defined by
\begin{align*}
\Psi_k(x, \mu)=\left(M_k(x) \mu^{(N-2) / 2}, \mu^{(N-2) / 2}\right)\Big(\int_{\mathbb{R}^N} U_{0,1}^{\frac{N+2}{N-2}}\Big)^2- \left( \mu^k,\mu^k\right)\int_{\mathbb{R}^N} U_{0,1}^2, 
\end{align*}
where $\mu^{(N-2) / 2}:=\big(\mu_1^{(N-2) / 2}, \cdots, \mu_k^{(N-2) / 2}\big)^T$ and the matrix $M_k(x)=\big(m_{i j}(x)\big)_{1 \leq i, j \leq k}$ is defined by
\begin{align*}
	m_{i i}(x)=R\left(x_i\right), \quad  m_{i j}(x)=-G\left(x_i, x_j\right), \text { if } i \neq j .
\end{align*}
Here, $R(\cdot)$ and $G(\cdot,\cdot)$ denotes the Robin function  and Green function  of $\Omega$, which  are defined in Appendix A, respectively.

Set
\begin{align*}
	\mathcal{Q}_{\Omega,k}=\big\{(a^k,\mu^k)\in \Omega^k \times\left(\mathbb{R}^{+}\right)^k: \nabla_x \Psi_k\left(a^k,\mu^k\right)=0, \nabla_\lambda \Psi_k\left(a^k, \mu^k\right)=0\big\} .
	\end{align*}
    Then, we define a class of bounded  domain, on which  we seek $k$-spike solutions of \eqref{equ:Main-Problem} and prove their local uniqueness.
\begin{definition}
    We call a bounded domain $\Omega\subset\mathbb{R}^N$ admissible, provided that $\exists\ (a^k,\mu^k)\in \mathcal{Q}_{\Omega,k}$ for some $k\in \mathbb{N}^+$, such that $M_k\left(a^k\right)$ is a positive matrix and $\left(a^k, \mu^k\right)$ is a nondegenerate critical point of $\Psi_k$.
\end{definition}

Our first main result addresses the existence of $k$-spike solutions of \eqref{equ:Main-Problem}.
\begin{theorem}\label{th-3}
	Let $N \geq 6$ and  $\Omega\subset\mathbb{R}^N$ be an admissible domain. Suppose that  $ (a^k,\mu^k)\in \mathcal{Q}_{\Omega,k}$ staifies  $M_k\left(a^k\right)$ is a positive matrix and $\left(a^k, \mu^k\right)$ is a nondegenerate critical point of $\Psi_k$. Then  there exists a small $\rho_0>0$, such that for $\rho\in(0,\rho_0)$, \eqref{equ:Main-Problem} has a normalized  multispike solution $u_{\rho}$ with an associated Euler-Lagrange multiplier $\lambda_\rho$,
satisfying 
\begin{align}\label{equ:241120-e1}
    \left|\nabla u_{\rho}\right|^2 \rightharpoonup \mathcal{S}^{N / 2} \sum_{i=1}^k \delta_{a_i} \quad \text{as} \quad \rho \rightarrow 0^+,
\end{align}
and 
		\begin{align}\label{eq1.11}
			u_{\rho}(x)=	\sum_{j=1}^k P U_{x_{j,\rho}, \mu_{j,\rho}}+w_{\rho}, \text{ with }\left\|w_{\rho}\right\|=o_\rho(1) \text { and }  w_{\rho} \in \bigcap_{j=1}^k E_{x_{j,\rho}, \mu_{j,\rho}}. 
			\end{align} 
        Moreover, let $\mu_{j,\rho}=\left(u_{\rho}(x_{j,\rho})\right)^{\frac{2}{N-2}}$ for $j=1, \cdots, k, $, then as $\rho\to0^+$,
\begin{align}\label{eq1.12}
	x_{j,\rho} \rightarrow a_j, ~\rho^\frac{1}{2}\mu_{j,\rho} \to \Big(\int_{\mathbb{R}^N} U_{0,1}^2\sum_{i=1}^k \mu_i^{-2}\Big)^\frac12\mu_j,~\rho^\frac{4-N}{2}\lambda_\rho\to \Big(\int_{\mathbb{R}^N} U_{0,1}^2\sum_{i=1}^k \mu_i^{-2}\Big)^\frac{4-N}{2}.
\end{align}

\end{theorem} 

We establish  the normalized $k$-spike solution  in Theorem \ref{th-3}  based on the framework of \cite{CLP,MA,Rey}, where   the spike solution for the following classical   Br\'{e}zis-Nirenberg problem were investigated:
	\begin{equation}\lab{equ:original-Problem}
		\begin{cases}
			-\Delta u=|u|^{2^*-2} u +\lambda u\quad \text { in } \Omega\subset{\mathbb{R}^N}, \\
			u>0, \quad 	u \in H_0^1(\Omega).
		\end{cases}
	\end{equation}
It follows from the pioneering paper \cite{BN}  that  \eqref{equ:original-Problem}  has a positive solution for every $\lambda\in(0,\lambda_1(\Omega))$ and $N\geq 4$. 
In \cite{Rey},  Rey showed that if a solution $u_{\lambda}$ of \eqref{equ:original-Problem} satisfies
	\begin{align}\label{equ:241113-e1}
			\left|\nabla u_{\lambda}\right|^2 \rightharpoonup \mathcal{S}^{N / 2} \delta_{x_0}, \text { as } \lambda \rightarrow 0^+,
	\end{align}
	then $x_0$ must be a critical point of $R(x)$. Conversely if $x_0$ is a nondegenerate critical point of $R(x)$ and $N \geq 5$, then \eqref{equ:original-Problem} has a solution $u_{\lambda}$ satisfying \eqref{equ:241113-e1}.  Han in \cite{HZC}  established the asymptotic expansion for  the positive solutions of \eqref{equ:original-Problem}.  Besides, Glangetas \cite{LG} studies the  uniqueness  of positive solutions.  Musso and Pistoia in Ref.\cite{MA} show that,
	 if $\left(a^k, \mu^k\right)$ is a nondegenerate critical point of $\Psi_k$,  then there exists a family of solution  $u_{\lambda}$ to \eqref{equ:original-Problem} satisfying 
	\begin{equation}\label{1.4}
		\left|\nabla u_{\lambda}\right|^2 \rightharpoonup \mathcal{S}^{N / 2} \sum_{i=1}^k \delta_{a_i}~ \text { as }~\lambda \rightarrow 0^+,
	\end{equation}
	which blows up and concentrates at the points $a_1, \cdots, a_k$ with rates of concentration 
	$ \mu_{1,\lambda} , \ldots,  \mu_{k,\lambda}$ such that $\lim\limits_{\lambda \rightarrow 0}\mu_{\lambda} \lambda^{1 /(N-4)} \rightarrow$ $\mu_j$, for $j=1, \cdots, k.$
    Furthermore, Cao, Luo and Peng \cite{CLP} obtained the structure of solutions satisfying \eqref{1.4} and showed that $k$-spike solution of \eqref{equ:original-Problem} is locally  unique  by using various local Pohozaev identities and blow-up analysis.

\begin{remark} 

 In  \cite[Theorem 11]{BHG3},  Noris, Tavares and Verzini proved that  if
$0<\rho \leqslant  \rho^*$ $:= \frac{2}{N} (\frac{1}{\mathcal{S}})^{\frac{N-2}{2}}(\lambda_1(\Omega)) ^{-1}$
then the problem
\begin{equation*}
\left\{\begin{array}{ll}
	-\Delta u+\omega u= |u|^{2^*-1} \\
	\int_{\Omega} u^2=\rho, \quad u \in H_0^1(\Omega)
\end{array}\right.
\end{equation*}
admits a positive solution $u_1$ for a suitable $\omega \in\left(-\lambda_1(\Omega), 0\right)$, which is a local minimizer of the associated energy. Moreover, there exists $\alpha\geq \lambda_1(\Omega) $ such that	  $$
\int_{\Omega}|\nabla u_1|^2\leq \rho \alpha\to 0 \quad \text{as } \quad \rho \to 0^+.
$$ This means that the local minimizer $u_1$ is distinct from  the $k$-spike solution obtained in our Theorem \ref{th-3}, since the $k$-spike solution satisfies $$\int_{\Omega}|\nabla u_{\rho}|^2dx \to k \mathcal{S}^{N / 2} \quad \text{as } \quad \rho \to 0^+.$$

\end{remark}


\begin{remark} 

Pierotti, Verzini, and Yu in \cite{DGY} generalized several results from \cite[Theorem 11]{BHG3} and established the existence of a second positive solution for $0<\rho<\rho^*$
  in star-shaped domains. This solution is of mountain pass type. We also note that Song and Zou \cite{SZ} obtained two distinct positive solutions in star-shaped domains, corresponding to a local minimizer and a mountain pass solution, respectively. In contrast to these results, our approach does not require the assumption of star-shaped domains. Furthermore, for $k\ge2$, the normalized $k$-spike solutions in Theorem \ref{th-3} are high-energy solutions, which generally cannot be classified as either local minimizers or mountain pass solutions. Additionally, we will prove that the normalized multispike solutions are locally unique.

\end{remark}

In what follows, we investigate the local uniqueness and count the number of $k$-spike type solutions. To this end, for $a^k=$ $\left(a_1, \cdots, a_k\right)\in \Omega^k$,   we introduce 
\begin{align*}
	\mathcal{Q}_{a^k}:=\big\{\mu^k=\left(\mu_1, \cdots, \mu_k\right)\in \big(\mathbb{R^+}\big)^k: \left(a^k,\mu^k\right) \in \mathcal{Q}_{\Omega,k}\big\} .
	\end{align*}
Our result can be stated as follows.

\begin{theorem}\label{th-4}
	For $N \geq 6$ and any given $a^k=\left(a_1, \cdots, a_k\right)$, suppose that $M_k\left(a^k\right)$ is a positive matrix and $\left(a^k, \mu^k\right)$ is a nondegenerate critical point of $\Psi_k$. Then there exists small $\rho_0>0$, such that  for any  $\rho\in (0,\rho_0)$, the normalized multispike $u_{\rho}(x)$ satisfying \eqref{eq1.11} and \eqref{eq1.12} is unique.
    Moreover, if $\left(a^k,\mu^k\right)$ is always a nondegenerate critical point of $\Psi_k$ for any $\mu^k\in \mathcal{Q}_{a^k}$, then the number of solutions to \eqref{equ:Main-Problem} satisfying  \eqref{equ:241120-e1}  $=\sharp \mathcal{Q}_{a^k}$,
	where $\sharp \mathcal{Q}_{a^k}$ is the number of the elements in the set $\mathcal{Q}_{a^k}$.
\end{theorem}
   We adopt some  novel strategy based on the local Pohozaev identity method to  study the local uniqueness of spike solutions of Brézis-Nirenberg problem under the constraint of prescribed mass. The local Pohozaev identity method was introduced by \cite{Grossi} to calculate the number of single-peak solutions for the Schr\"odinger equation.  Whereafter, this method was developed to study the uniqueness of peak solutions for some different equations, see \cite{CLP,Grossi,GHLY,GLWZ,GLW, LPWY} and the references therein.

This paper is organized as follows. In Section 2, we will prove Theorem \ref{th-3}, while in Section 3, we prove the local uniqueness of the normalized  solutions.  In section 4,  we give the defination of Green's function and some useful estimates.

\section{Existence of normalized solutions}

\begin{proposition}\label{Th-2} 
	Let $N \geq 5$  and  assume that $ (a^k,\mu^k)\in \mathcal{Q}_{\Omega,k}$ is a nondegenerate critical point of $\Psi_k$, then there exists $\lambda_0>0$ small such that  for all $\lambda\in(0,\lambda_0]$, equation \eqref{equ:original-Problem} has a solution $u_{\lambda}(x)$ satisfying \eqref{1.4} and and $M_k\left(a^k\right)$ is a nonnegative matrix. $u_{\lambda}(x)$ can be written as
	\begin{align}\label{equ:241113-e30}
		u_{\lambda}=\sum_{j=1}^k P U_{x_{j,\lambda}, \mu_{j,\lambda}}+w_{\lambda}, 
	\end{align}
	satisfying, for $j=1, \cdots, k, \mu_{j,\lambda}:=\left(u_{\lambda}(x_{j,\lambda})\right)^{\frac{2}{N-2}}$,
	\begin{align}\label{2.2}
		x_{j,\lambda} \rightarrow a_j, ~\mu_{j,\lambda} \rightarrow+\infty,~\left\|w_{\lambda}\right\|=o(1) ~\text { and } ~w_{\lambda} \in \bigcap_{j=1}^k E_{x_{j,\lambda}, \mu_{j,\lambda}}.
	\end{align}
	Moreover, assume additionally that $M_k\left(a^k\right)$ is  positive, then up to a subsequence, 
	\begin{equation}\label{equ:241108-e5}
	\lim _{\lambda \rightarrow 0}\left(\lambda^{\frac{1}{N-4}} \mu_{j, \lambda}\right)^{-1}=\mu_j, \text { for } j=1, \cdots, k,
\end{equation}
	and 
	\begin{equation}\label{equ:241014-e5}
		\big|x_{j,\lambda}-a_j\big|=\left\{\begin{array}{ll}
			O(\frac{1}{\mu_{\lambda}}), & \text { if } N=5, \vspace{2mm}\\
			O(\frac{1}{\mu_{\lambda}^2}), & \text { if } N \geq 6,
		\end{array}, \ \big |\mu_j-(\lambda^{\frac{1}{N-4}} \mu_{j,\lambda})^{-1}\big|= \begin{cases}O(\frac{1}{\mu_{\lambda}}), & \text { if } N=5,  \vspace{2mm} \\
			O(\frac{1}{\mu_{\lambda}^2}), & \text { if } N \geq 6 .\end{cases}\right.
	\end{equation}
    Especially, the solution $u_\lambda$ satisfying the estimates of \eqref{equ:241113-e30}-\eqref{equ:241108-e5} is unique  provided  $N\geq6$.
\end{proposition}
\begin{proof}
    It has been proved by \cite{MA} that equation \eqref{equ:original-Problem} has a solution $u_{\lambda}(x)$ satisfying \eqref{1.4}.  Other conclusions can be found in Theorems 1.1-1.2 and Proposition 2.5  of \cite{CLP}.
\end{proof}


\noindent{ \bf{Proof of Theorem \ref{th-3}}. } 
From Proposition \ref{Th-2}, we  know that there are small $\lambda_0>0$, such that for $\lambda \in (0,\lambda_0] \subset  (0,\lambda_1)$, the problem \eqref{equ:original-Problem} has a unique solution $u_\lambda$ satisfies 
\begin{equation}\label{eq2.5}
		\left|\nabla u_{\lambda}\right|^2 \rightharpoonup \mathcal{S}^{N / 2} \sum_{i=1}^k \delta_{a_i}~ \text { as }~\lambda \rightarrow 0^+,
	\end{equation}
and 
\begin{align}\label{equ:241014-e3}
	u_\lambda= \sum_{j=1}^{k} P U_{x_{j,\lambda}, \mu_{j,\lambda}}+\omega_\lambda,
\end{align}
and $u_\lambda$ satisfies \eqref{2.2} and \eqref{equ:241108-e5}. Moreover, from Proposition A.2. in \cite{CLP}, we obtain that 
\begin{align}\label{equ:240914-e5}
	\left\|w_{\lambda}\right\|= \begin{cases}
		\vspace{2mm} O\big(\frac{(\log \mu_{\lambda}^2)^{2 / 3}}{\mu_{\lambda}^4}+\frac{\lambda  ( \log \mu_{\lambda} )^{2 / 3}}{\mu_{\lambda}^2}\big), & \text { if } N=6, \\ O\big(\frac{1}{\mu_{\lambda}^{(N+2) / 2}}+\frac{\lambda}{\mu_{\lambda}^2}\big), & \text { if } N>6,\end{cases}
\end{align}
where $\mu_\lambda:=\min\{\mu_{1,\lambda},...,\mu_{k,\lambda}\}$.
Let
\begin{align*}
	v_\rho=\frac{  \sqrt{\rho } u_\lambda  }{  \left\|u_\lambda\right\|_{L^2\left(\Omega\right)}  },
\end{align*}
then $ 	v_\rho $ satisfies the equation
\begin{equation*}
	\begin{cases}
		-\Delta 	v_\rho=\lambda	v_\rho+ \left(  \frac{\left\|u_\lambda\right\|_{L^2\left(\Omega\right)}}{\sqrt{\rho }} \right)^{\frac{N-2}{4}}    	v_\rho^{2^*-1}  \quad \text { in } \Omega, \\
		\int_{\Omega} v^2_\rho dx=\rho, \quad 	v_\rho>0, \quad 		v_\rho \in H_0^1(\Omega) .
	\end{cases}
\end{equation*}
Therefore, to show that $v_\rho $ satisfies \eqref{equ:Main-Problem}, it remains  to prove that there exists 
\begin{equation}\label{eq2.8}
    \text{$\lambda=\lambda_\rho\in (0,\lambda_0]$, such that  $\left\|u_{\lambda\rho}\right\|^2_{L^2\left(\Omega\right)} ={\rho }$.}
\end{equation}
 Indeed, by the uniqueness of the solution $u_\lambda, \ \lambda\in (0,\lambda_0]$, one can easily show that $\left\|u_\lambda\right\|_{L^2\left(\Omega\right)}$ is continuous w.r.t. $\lambda\in (0,\lambda_0]$. 
Moreover, it follows from  \eqref{equ:241113-e30} that 
\begin{equation}\label{eq2.9}
    \begin{aligned}
		\left\|u_\lambda\right\|^2_{L^2\left(\Omega\right)} & = \int_{\Omega} \Big( \sum_{j=1}^{k} P U_{x_{j,\lambda}, \mu_{j,\lambda}}+\omega_\lambda  \Big)^2 dx\\
		& = \int_{\Omega}\sum_{j=1}^{k}  \left( P U_{x_{j,\lambda}, \mu_{j,\lambda}}  \right)^2  \\    
		&	+O\Big(  \sum_{i,j=1}^{k} \int_{\Omega}  P U^2_{x_{j,\lambda}, \mu_{j,\lambda}}    +\sum_{j=1}^{k}  \left\|P U_{x_\lambda,\mu_\lambda}\right\|_{L^2(\Omega)}\left\|\omega_\lambda\right\|_{L^2(\Omega)}+\left\|\omega_\lambda\right\|_{L^2(\Omega)}^2\Big) \\
		& = \sum_{j=1}^{k}\frac{1}{ \mu^2_{j,\lambda}} \int_{\mathbb{R}^N} U^2_{0,,1}dx+O\left(  \frac{1}{\mu^{\frac{N}{2}}_\lambda} + \frac{1}{\mu^2_\lambda}\left\|\omega_\lambda\right\|_{L^2(\Omega)}^2+ \left\|\omega_\lambda\right\|_{L^2(\Omega)}^2\right)\\
        &\to0^+ \text{ as } \lambda\to0^+.
	\end{aligned}
\end{equation}
Set $\rho_0:=\max_{\lambda\in(0,\lambda_0]}\left\|u_\lambda\right\|^2_{L^2\left(\Omega\right)}$, then for any $\rho<\rho_0$, there exists at least one $\lambda_\rho\in(0,\lambda_0]$  such that \eqref{eq2.8} holds. Hence, $v_\rho $ satisfies \eqref{equ:Main-Problem}. 

From \eqref{eq2.8} and \eqref{eq2.9} we can see that 
\begin{align*}
		\lim\limits_{\rho \to 0} \sum\limits_{j=1}^{k}\frac{1}{\rho \mu^2_{j,\lambda_\rho}}= \big(\int_{\mathbb{R}^N} U^2_{0,,1}dx\big)^{-1}.
	\end{align*}
This together with \eqref{2.2} \eqref{equ:241108-e5} indicates \eqref{eq1.12}. Using \eqref{eq2.5}, \eqref{equ:241014-e3}, and \eqref{equ:240914-e5}  we further obtain \eqref{equ:241120-e1} and \eqref{eq1.11}.
The proof of Theorem  \ref{th-3} is complete.
\qed

\section{The local uniqueness of normalized problem}
	
To obtain the  local uniqueness of such type of solutions, we need to estimate the
difference between two solutions concentrating at the same points.

Let $u_{\rho}^{(1)}(x), u_{\rho}^{(2)}(x)$ be two different solutions of \eqref{equ:Main-Problem} satisfying \eqref{equ:241120-e1}. Under the assumption that $M_k\left(a^k\right)$ is a positive matrix, we find from Theorem \ref{th-3} that $u_{\rho}^{(l)}(x)$ can be written as
\begin{equation*}
	u_{\rho}^{(l)}=\sum_{j=1}^k P U_{x_{j, \rho}^{(l)}, \mu_{j, \rho}^{(l)}}+w_{\rho}^{(l)},\ \ l=1,2,
\end{equation*}
satisfying, for $j=1, \cdots, k,  ~\mu_{j,\rho}^{(l)}=\big(u_{\rho}(x_{j,\rho}^{(l)})\big)^{\frac{2}{N-2}}$,

\begin{align*}
	x_{j,\rho}^{(l)} \rightarrow a_j,~\big(\rho^{\frac{1}{N-4}} \mu_{j,\rho}^{(l)}\big)^{-1} \rightarrow \mu_j,~\|w_{\rho}^{(l)}\|=o(1)~ \text { and } ~w_{\rho}^{(l)} \in \bigcap_{j=1}^k E_{x_{j,\rho}^{(l)}, \mu_{j,\rho}^{(l)}} .
\end{align*}

 Similar to \eqref{equ:240914-e5}, we also have
\begin{align}\label{equ:241114-e4}
	\left\| w_{\rho}^{(l)}  \right\|= \begin{cases}
		O\big(\frac{(\log \bar{\mu}_{\rho}^2)^{2 / 3}}{\bar{\mu}_{\rho}^4}+\frac{\lambda_\rho  ( \log \mu_{\rho} )^{2 / 3}}{\bar{\mu}_{\rho}^2}\big), & \text { if } N=6, \\ O\big(\frac{1}{\bar{\mu}_{\rho}^{(N+2) / 2}}+\frac{\lambda_\rho}{\bar{\mu}_{\rho}^2}\big), & \text { if } N>6,\end{cases}
\end{align}
where $\bar{\mu}_\rho:=\min \left\{\mu_{1, \rho}^{(1)}, \cdots, \mu_{k, \rho}^{(1)}, \mu_{1, \rho}^{(2)}, \cdots, \mu_{k, \rho}^{(2)}\right\}$.  

Proceeding  similar arguments as those in the proofs for Propositions 3.2 and 3.3 in \cite{CLP}, we can obtain the following two Lemmas. For brevity, we omit their proofs here.  

\begin{lemma} For $N \geq 6$, there holds
	\begin{align*}
		\left\|w_{\rho}^{(1)}-w_{\rho}^{(2)}\right\|=o(\frac{1}{\bar{\mu}_{\rho}^{(N+2) / 2}}) .
	\end{align*}
\end{lemma} 


\begin{lemma}
	For $N \geq 6$ and $j=1, \cdots, k$, it holds 
	\begin{align}\label{equ:241015-e7}
		|x_{j,\rho }^{(1)}-x_{j, \rho }^{(2)}|=o\big(\frac{1}{\bar{\mu}_{\rho}^2}\big)
		\quad \text{and} \quad 
		|\mu_{j, \rho }^{(1)}-\mu_{j,\rho }^{(2)}|=o\big(\frac{1}{\bar{\mu}_{\rho}^2}\big).
	\end{align}
\end{lemma}
	

\begin{lemma} $($see \cite[Lemma 3.4]{CLP}$)$
	For any constant $0<\sigma \leq N-2$, there is a constant $C>0$, such that
	\begin{align}\label{equ:241009-e1}
		\int_{\mathbb{R}^N} \frac{1}{|y-z|^{N-2}} \frac{1}{(1+|z|)^{2+\sigma}} d z \leq \begin{cases}C(1+|y|)^{-\sigma}, & \sigma<N-2, \\ C|\ln | y| |(1+|y|)^{-\sigma}, & \sigma=N-2 .\end{cases}
	\end{align} 
\end{lemma}

	Now we set
\begin{align}\label{equ:241009-e2}
		\xi_\rho(x)=\frac{u_\rho^{(1)}(x)-u_\rho^{(2)}(x)}{\big\|u_\rho^{(1)}-u_\rho^{(2)}\big\|_{L^{\infty}(\Omega)}},
\end{align}
then $\xi_\rho(x)$ satisfies $\left\|\xi_\rho\right\|_{L^{\infty}(\Omega)}=1$ and
	
\begin{align}\label{equ:241009-e3}
	-\Delta \xi_\rho(x)=\lambda_\rho^{(1)} \xi_\rho(x)+C_\rho(x) \xi_\rho(x)+g_\rho(x),
\end{align}
	where
\begin{align*}
		C_{\rho}(x):=\left(\frac{N+2}{N-2}\right) \int_0^1\left(t u_\rho^{(1)}(x)+(1-t) u_\rho^{(2)}(x)\right)^{\frac{4}{N-2}} d t 
\end{align*}
and 
\begin{align*}
		g_\rho(x):=\frac{\lambda_\rho^{(1)}-\lambda_\rho^{(2)}}{\| u_\rho^{(1)}- u_\rho^{(2)} \|_{L^{\infty}}} u_\rho^{(2)}(x).
	\end{align*}

Now, let $\xi_{\rho, j}(x)=\xi_\rho\big(\frac{x}{\mu_{j, \rho}^{(1)}}+x_{j, \rho}^{(1)}\big)$ for $j=1, \cdots, k$.   Since $\xi_{\rho, j}(x)$ is bounded, by the regularity theory in \cite{GT}, we find

$$
\xi_{\rho, j}(x) \in C^{1, \alpha}\left(B_r(0)\right) \text { and }\left\|\xi_{\rho, j}\right\|_{C^{1, \alpha}}\left(B_r(0)\right) \leq C
$$
for any fixed large $r$ and $\alpha \in(0,1)$ if $\rho$ and $\lambda_\rho$ is small, where the constants $r$ and $C$ are independent of $\rho$ and $j$. So we may assume that $\xi_{\rho, j}(x) \rightarrow \xi_j(x)$ in $C\left(B_r(0)\right)$. By direct calculations, we know
\begin{align}\label{equ:241015-e1}
	-\Delta \xi_{\rho, j}(x)&=-\frac{1}{(\mu_{j, \rho}^{(1)})^2} \Delta  \xi_\rho(\frac{x}{\mu_{j, \rho}^{(1)}}+x_{j, \rho}^{(1)})\\ \nonumber
	&=\frac{1}{(\mu_{j, \rho}^{(1)})^2} C_\rho(\frac{x}{\mu_{j, \rho}^{(1)}}+x_{j, \rho}^{(1)}) \xi_{\rho, j}(x)+\frac{ \lambda_\rho^{(1)} }{(\mu_{j, \rho}^{(1)})^2} \xi_{\rho, j}(x) +\frac{1}{(\mu_{j, \rho}^{(1)})^2}g_\rho(\frac{x}{\mu_{j, \rho}^{(1)}}+x_{j, \rho}^{(1)}).
\end{align}

\begin{lemma}\label{lm:241114-l1}
	For $N \geq 6$ and   any given $\Phi(x) \in C_0^{\infty}\left(\mathbb{R}^N\right)$ with $\operatorname{supp} \Phi(x) \subset B_{\mu_{j, \rho}^{(1)}d}\big(x_{j, \rho}^{(1)}\big)$, it holds
	\begin{equation}\label{equ:241114-e2}
		\begin{aligned}
			\frac{1}{(\mu_{j, \rho}^{(1)})^2} \int_{B_{\mu_{j, \rho}(1)}(x_{j, \rho}^{(1)})} & C_{\rho}(\frac{x}{\mu_{j, \rho}^{(1)}}+x_{j, \rho}^{(1)}) \xi_{\rho, j}(x) \Phi(x) d x \\
			& =(\frac{N+2}{N-2}) \int_{\mathbb{R}^N} U_{0,1}^{\frac{4}{N-2}}(x) \xi_{\rho, j}(x) \Phi(x) d x+o(\frac{1}{\bar{\mu}_{\rho}}),
		\end{aligned}
	\end{equation}
	and
	\begin{align}\label{equ:241114-e3}
		&	\frac{1}{(\mu_{j, \rho}^{(1)})^2}\int_{B_{\mu_{j, \rho}^{(1)} d}(x_{j, \rho}^{(1)})}g_\rho(\frac{x}{\mu_{j, \rho}^{(1)}}+x_{j, \rho}^{(1)})\Phi(x) dx \\ \nonumber
		&=\frac{(\mu_{j, \rho}^{(1)})^{(N-2)/2}}{(\mu_{j, \rho}^{(1)})^2\rho} \int_{B_{\mu_{j, \rho}^{(1)} d}(x_{j, \rho}^{(1)})}  U_{0, 1}(x)\Phi(x)dz \cdot \Big(-\frac{4}{N-2} \sum_{j=1}^k 	\frac{1}{(\mu_{j, \rho}^{(1)})^{\frac{N-2}{2}}}\int_{\mathbb{R}^N} \big [    U_{0, 1}(x)  \big]^{\frac{N+2}{N-2}}    \xi_{\rho, j}(x)   dx  \Big) +o(\frac{1}{\bar{\mu}_{\rho}}).
	\end{align}

\end{lemma}

\begin{proof}

Now, we estimate $C_\rho(\frac{x}{\mu_{j, \rho}^{(1)}}+x_{j, \rho}^{(1)}) $. Similar to \eqref{equ:241014-e5}, if $N \geq 6$, we have
\begin{equation*}
		\big|x_{j,\rho}-a_j\big|=O(\frac{1}{\bar{\mu}_{\rho}^2}) \quad  \text{and} \quad 
         \big |\mu_j-(\lambda_\rho^{\frac{1}{N-4}} \mu_{j,\rho})^{-1}\big|= 
			O(\frac{1}{\bar{\mu}_{\rho}^2}).
	\end{equation*}.
Then, we notice that 
\begin{equation}\label{equ:241015-e6}
	\begin{aligned}
		& U_{x_{j, \rho}^{(1)}, \mu_{j, \rho}^{(1)}}(x)-U_{x_{j, \rho}^{(2)}, \mu_{j, \rho}^{(2)}}(x) \\
		& =O\left(|x_{j, \rho}^{(1)}-x_{j, \rho}^{(2)}| \cdot(\nabla_y U_{y, \mu_{j, \rho}^{(1)}}(x)\big|_{y=x_{j, \rho}^{(1)}})+|\mu_{j, \rho}^{(1)}-\mu_{j, \rho}^{(2)}| \cdot(\nabla_\mu U_{x_{j, \rho}^{(1)}, \mu}(x)\big|_{\mu=\mu_{j, \rho}^{(1)}})\right) \\
		& =O\left(\mu_{j, \rho}^{(1)}|x_{j, \rho}^{(1)}-x_{j, \rho}^{(2)}|+(\mu_{j, \rho}^{(1)})^{-1}|\mu_{j, \rho}^{(1)}-\mu_{j, \rho}^{(2)}|\right) U_{x_{j, \rho}^{(1)}, \mu_{j, \rho}^{(1)}}(x)=o(\frac{1}{\bar{\mu}_{\rho}}) U_{x_{j, \rho}^{(1)}, \mu_{j, \rho}^{(1)}}(x).
	\end{aligned}
\end{equation}

which means
\begin{align}\label{equ:241031-e2}
	u_{\rho}^{(1)}(x)-u_{\rho}^{(2)}(x)=o\big(\frac{1}{\bar{\mu}_{\rho}}\big)\big(\sum_{j=1}^k U_{x_{j, \rho}^{(1)}, \mu_{j, \rho}^{(1)}}(x)\big)+O\big(\sum_{l=1}^2|w_{\rho}^{(l)}(x)|\big).
\end{align}

Similar to \eqref{equ:241007-e1} and \eqref{equ:241114-e5},  we also have
\begin{align}\label{equ:250305-e1}
		P U_{x^{(l)}_{j,\rho}, \mu^{(l)}_{j,\rho}}(x)=O\big(\frac{1}{(\mu^{(l)}_{j,\rho})^{(N-2) / 2}}\big) \text { in } C^1\big(\Omega \backslash B_d(x^{(l)}_{j,\rho})\big) \quad \text{and}  \quad  \varphi_{x^{(l)}_{j,\rho}, \mu^{(l)}_{j,\rho}}(x)=O\big(\frac{1}{(\mu^{(l)}_{j,\rho})^{(N-2) / 2}}\big) \text { in }C^1(\Omega).
 	\end{align}
Then for a small fixed $d$ and $x \in B_d(x_{j, \rho}^{(1)})$, by \eqref{equ:250305-e1}, we find
\begin{align}\label{equ:241106-e4}
	C_{\rho}(x)=\big(\frac{N+2}{N-2}+o(\frac{1}{\bar{\mu}_{\rho}})\big) U_{x_{j, \rho}^{(1)}, \mu_{j, \rho}^{(1)}}^{\frac{4}{N-2}}(x)+O\big(\frac{1}{\bar{\mu}_{\rho}^2}\big)+O\big(\sum_{l=1}^2\left|w_{\rho}^{(l)}(x)\right|^{\frac{4}{N-2}}\big).
\end{align}

Next, for any given $\Phi(x) \in C_0^{\infty}\left(\mathbb{R}^N\right)$ with $\operatorname{supp} \Phi(x) \subset B_{\mu_{j, \rho}^{(1)}d}\big(x_{j, \rho}^{(1)}\big)$, we have

\begin{equation}\label{equ:241015-e2}
\begin{aligned}
	\frac{1}{(\mu_{j, \rho}^{(1)})^2} \int_{B_{\mu_{j, \rho}(1)}(x_{j, \rho}^{(1)})} & C_{\rho}(\frac{x}{\mu_{j, \rho}^{(1)}}+x_{j, \rho}^{(1)}) \xi_{\rho, j}(x) \Phi(x) d x \\
	& =(\frac{N+2}{N-2}) \int_{\mathbb{R}^N} U_{0,1}^{\frac{4}{N-2}}(x) \xi_{\rho, j}(x) \Phi(x) d x+o(\frac{1}{\bar{\mu}_{\rho}}).
\end{aligned}
\end{equation}
Thus, we obtain \eqref{equ:241114-e2}.
Also from the fact that $\left\|\xi_{\rho}\right\|_{L^{\infty}(\Omega)}=1$, we know
\begin{align}\label{equ:241015-e3}
	\frac{\lambda_\rho^{(1)}}{(\mu_{j, \rho}^{(1)})^2} \int_{B_{\mu_{j, \rho}^{(1)} d}(x_{j, \rho}^{(1)})} \xi_{\rho, j}(x) \Phi(x) d x=o(\frac{1}{\bar{\mu}_\rho}).
\end{align}

Next, we estimate $g_\rho(\frac{x}{\mu_{j, \rho}^{(1)}}+x_{j,\lambda}^{(1)})$. 
From \eqref{equ:Main-Problem}, we find that
\begin{align*}
	\int_{\Omega} | \nabla u_\rho^{(1)} |^2dx-\lambda_\rho^{(1)} \int_{\Omega} (u_\rho^{(1)})^2dx=\int_{\Omega} (u_\rho^{(1)})^{\frac{2N}{N-2}} dx, 
\end{align*}
which implies that 
\begin{align*}
	\frac{ \rho (\lambda_\rho^{(1)}-\lambda_\rho^{(2)})}{\| u_\rho^{(1)}- u_\rho^{(2)} \|_{L^{\infty}}}= \int_{\Omega}  \nabla ( u_\rho^{(1)} + u_\rho^{(2)} ) \nabla \xi_\rho dx + \int_{\Omega} 	\frac{  (u_\rho^{(2)})^{\frac{2N}{N-2}}-(u_\rho^{(1)})^{\frac{2N}{N-2}} }{\| u_\rho^{(1)}- u_\rho^{(2)} \|_{L^{\infty}}}dx.
\end{align*}

Since 
\begin{equation*}
	\int_{\Omega}\left(u_{\rho}^{(1)}+u_{\rho}^{(2)}\right) \xi_{\rho}dx=\frac{1}{\|u_{\rho}^{(1)}-u_{\rho}^{(2)}\|_{L^{\infty}\left(\mathbb{R}^N\right)}}\big[\int_{\Omega}\left(u_{\rho}^{(1)}\right)^2dx-\int_{\mathbb{R}^N}\left(u_{\rho}^{(2)}\right)^2dx\big]=0
\end{equation*}
and 
\begin{equation*}
	\begin{aligned}
			& \left(u_{\rho}^{(1)}\right)^{\frac{2N}{N-2}}-\left(u_{\rho}^{(2)}\right)^{\frac{2N}{N-2}} \\
			= & \left(u_{\rho}^{(1)}\right)^{\frac{2N}{N-2}}-\left(u_{\rho}^{(1)}\right)^{\frac{2N}{N-2}-1} u_{\rho}^{(2)}+\left(u_{\rho}^{(1)}\right)^{\frac{2N}{N-2}-1} u_{\rho}^{(2)}-\left[\left(u_{\rho}^{(2)}\right)^{\frac{2N}{N-2}}-\left(u_{\rho}^{(2)}\right)^{\frac{2N}{N-2}-1} u_{\rho}^{(1)}+\left(u_{\rho}^{(2)}\right)^{\frac{2N}{N-2}-1} u_{\rho}^{(1)}\right] \\
			= & \left(u_{\rho}^{(1)}\right)^{\frac{2N}{N-2}-1}\left(u_{\rho}^{(1)}-u_{\rho}^{(2)}\right)+\left(u_{\rho}^{(2)}\right)^{\frac{2N}{N-2}-1}\left(u_{\rho}^{(1)}-u_{\rho}^{(2)}\right)+u_{\rho}^{(1)} u_{\rho}^{(2)}\left[\left(u_{\rho}^{(1)}\right)^{\frac{2N}{N-2}-2}-\left(u_{\rho}^{(2)}\right)^{\frac{2N}{N-2}-2}\right].
		\end{aligned}
\end{equation*}

Then we can deduce that 
\begin{align*}
	\frac{ \rho (\lambda_\rho^{(1)}-\lambda_\rho^{(2)}  )}{\| u_\rho^{(1)}- u_\rho^{(2)} \|_{L^{\infty}}}=&-\lambda_\rho^{(2)}	\int_{\mathbb{R}^N}\left(u_{\rho}^{(1)}+u_{\rho}^{(2)}\right) \xi_{\rho}dx+ \int_{\Omega}  \nabla ( u_\rho^{(1)} + u_\rho^{(2)} ) \nabla \xi_\rho dx\\ \nonumber
	&-\int_{\Omega}\left[ (u_\rho^{(1)} )^{\frac{N+2}{N-2}}\xi_{\rho}+ (u_\rho^{(2)} )^{\frac{N+2}{N-2}}\xi_{\rho}+  	\frac{ u_\rho^{(1)}u_\rho^{(2)} \Big((u_\rho^{(1)})^{\frac{4}{N-2}}-(u_\rho^{(2)})^{\frac{4}{N-2}}  \Big)}{\| u_\rho^{(1)}- u_\rho^{(2)} \|_{L^{\infty}}} \right] dx\\ \nonumber
	&= \int_{\Omega}( \lambda_\rho^{(1)}-\lambda_\rho^{(2)}   )u_\rho^{(1)}  \xi_{\rho} dx+ \frac{1}{\| u_\rho^{(1)}- u_\rho^{(2)} \|_{L^{\infty}}}  \int_{\Omega}   \Big(  u_\rho^{(1)} (u_\rho^{(2)})^{\frac{N+2}{N-2}}-  u_\rho^{(2)} (u_\rho^{(1)})^{\frac{N+2}{N-2}}     \Big)  dx\\ \nonumber
	&= \int_{\Omega}( \lambda_\rho^{(1)}-\lambda_\rho^{(2)}   )u_\rho^{(1)} \xi_{\rho}dx-\frac{4}{N-2}  \int_{\Omega} [u_\rho^{(2)}+\theta (u_\rho^{(1)}- u_\rho^{(2)}) ]^{(\frac{2N}{N-2}-3) }    u_\rho^{(1)} u_\rho^{(2)} \xi_{\rho}   dx.
\end{align*}

When $N\geq 7$, it follows from  \eqref{equ:241015-e6}  and  \eqref{equ:241031-e2}  that 
\begin{align}\label{equ:241031-e4}
	\int_{\Omega} [u_\rho^{(2)}+\theta (u_\rho^{(1)}- u_\rho^{(2)}) ]^{(\frac{2N}{N-2}-3) }    u_\rho^{(1)} u_\rho^{(2)} \xi_{\rho}   dx   &=  
	\int_{\Omega} \big [   \sum_{j=1}^k U_{x_{j, \rho}^{(1)}, \mu_{j, \rho}^{(1)}}(x)  \big]^{\frac{N+2}{N-2}}    \xi_{\rho}   dx  + o( \frac{1}{\bar{\mu}_\rho}   ) 
\end{align}

When $N= 6$,  we have $\frac{2N}{N-2}=3$.  Similarly, we notice that 
\begin{align*}
&-\frac{4}{N-2} \int_{\Omega} [u_\rho^{(2)}+\theta (u_\rho^{(1)}- u_\rho^{(2)}) ]^{(\frac{2N}{N-2}-3) }    u_\rho^{(1)} u_\rho^{(2)} \xi_{\rho}   dx =	  -\frac{4}{N-2}	\int_{\Omega}     u_\rho^{(1)} u_\rho^{(2)} \xi_{\rho}   dx\\ \nonumber
&= 	\int_{\Omega} \big [   \sum_{j=1}^k U_{x_{j, \rho}^{(1)}, \mu_{j, \rho}^{(1)}}(x)  \big]^{2}    \xi_{\rho}   dx  + o( \frac{1}{\bar{\mu}_\rho}   ). 
\end{align*}

Besides,  we also notice 
\begin{align}\label{equ:241031-e5}
	\int_{\Omega}( \lambda_\rho^{(1)}-\lambda_\rho^{(2)}   )u_\rho^{(1)} dx=O\big( \lambda_\rho^{(1)}	\int_{\Omega} u_\rho^{(1)} dx \big)=O\big( \lambda_\rho^{(1)}	\int_{\Omega}(\sum_{j=1}^{k} PU_{x_{j, \rho}^{(1)}, \mu_{j, \rho}^{(1)}}(x) +\omega^{(1)}_\rho)  dx \big) = O(\frac{1}{\bar{\mu}_\rho^{(\frac{3N}{2}-5)}}).
\end{align}

Consequently, combing \eqref{equ:241031-e4} and \eqref{equ:241031-e5}, we see that 
\begin{align}\label{equ:241031-e3}
	\frac{ \rho (\lambda_\rho^{(1)}-\lambda_\rho^{(2)}  )}{\| u_\rho^{(1)}- u_\rho^{(2)} \|_{L^{\infty}}}&= -\frac{4}{N-2} \int_{\Omega} \big [   \sum_{j=1}^k U_{x_{j, \rho}^{(1)}, \mu_{j, \rho}^{(1)}}(x)  \big]^{\frac{N+2}{N-2}}    \xi_{\rho}   dx  + O(\frac{1}{\bar{\mu}_\rho^{(\frac{3N}{2}-5)}})\\ \nonumber
	& = -\frac{4}{N-2} \sum_{j=1}^k 	\frac{1}{(\mu_{j, \rho}^{(1)})^{\frac{N-2}{2}}}  \int_{\mathbb{R}^N} \big [    U_{0, 1}(x)  \big]^{\frac{N+2}{N-2}}    \xi_{\rho, j}(x)   dx +O(\frac{1}{\bar{\mu}_\rho^{(\frac{3N}{2}-5)}}).
\end{align}



By  \eqref{equ:241015-e6} and \eqref{equ:241031-e3}, for a small fixed $d$ and $x \in B_d(x_{j, \rho}^{(1)})$, we have

\begin{align*}
	&\frac{1}{(\mu_{j, \rho}^{(1)})^2}\int_{B_{\mu_{j, \rho}^{(1)} d}(x_{j, \rho}^{(1)})}g_\rho(\frac{x}{\mu_{j, \rho}^{(1)}}+x_{j, \rho}^{(1)})\Phi(x) dx 	\\ \nonumber
	&=\frac{1}{(\mu_{j, \rho}^{(1)})^2\rho}\int_{B_{\mu_{j, \rho}^{(1)} d}(x_{j, \rho}^{(1)})} \frac{ \rho (\lambda_\rho^{(1)}-\lambda_\rho^{(2)}  )}{\| u_\rho^{(1)}- u_\rho^{(2)} \|_{L^{\infty}}}   u_\rho^{(2)}(\frac{x}{\mu_{j, \rho}^{(1)}}+x_{j, \rho}^{(1)})\Phi(x)dx	\\ \nonumber
	& =  \frac{1}{(\mu_{j, \rho}^{(1)})^2\rho}  \Big(\int_{B_{\mu_{j, \rho}^{(1)} d}(x_{j, \rho}^{(1)})} \big(\sum_{j=1}^k U_{x_{j, \rho}^{(2)}, \mu_{j, \rho}^{(2)}}(\frac{x}{\mu_{j, \rho}^{(1)}}+x_{j, \rho}^{(1)})\big)\Phi(x)dx  + \int_{B_{\mu_{j, \rho}^{(1)} d}(x_{j, \rho}^{(1)})}  w_{\rho}^{(2)}(\frac{x}{\mu_{j, \rho}^{(1)}}+x_{j, \rho}^{(1)})\Phi(x)dx  \Big) \\ \nonumber
	&\cdot \Big( -\frac{4}{N-2} \sum_{j=1}^k 	\frac{1}{(\mu_{j, \rho}^{(1)})^{\frac{N-2}{2}}}\int_{\mathbb{R}^N} \big [    U_{0, 1}(x)  \big]^{\frac{N+2}{N-2}}    \xi_{\rho, j}(x)   dx   + O(\frac{1}{\bar{\mu}_\rho^{(\frac{3N}{2}-5)}}) \Big). 
	\end{align*}

By  \eqref{equ:241015-e7} and  \eqref{equ:241015-e6}, we have 
\begin{align} \label{equ:241015-e9}
	\int_{B_{\mu_{j, \rho}^{(1)} d}(x_{j, \rho}^{(1)})}    U_{x_{j, \rho}^{(2)}, \mu_{j, \rho}^{(2)}}(\frac{x}{\mu_{j, \rho}^{(1)}}+x_{j, \rho}^{(1)})\Phi(x)dx &= \int_{B_{\mu_{j, \rho}^{(1)} d}(x_{j, \rho}^{(1)})}   \big(1+o(\frac{1}{\bar{\mu}_{\rho}})\big) U_{x_{j, \rho}^{(1)}, \mu_{j, \rho}^{(1)}}(\frac{x}{\mu_{j, \rho}^{(1)}}+x_{j, \rho}^{(1)})\Phi(x)dx\\ \nonumber
	&=    (\mu_{j, \rho}^{(1)})^{(N-2)/2} \int_{B_{\mu_{j, \rho}^{(1)} d}(x_{j, \rho}^{(1)})}    U_{0, 1}(x)\Phi(x)dz +o(\frac{1}{\bar{\mu}_{\rho}}).
\end{align}

Also from \eqref{equ:241114-e4},  we calculate that 
\begin{align}\label{equ:241015-e8}
	&\frac{1}{(\mu_{j, \rho}^{(1)})^{N-2}}	\int_{B_{\mu_{j, \rho}^{(1)} d}(x_{j, \rho}^{(1)})}  w_{\rho}^{(2)}(\frac{x}{\mu_{j, \rho}^{(1)}}+x_{j, \rho}^{(1)})\Phi(x)dx=o(\frac{1}{\bar{\mu}_\rho}).
\end{align}

It follows from  \eqref{equ:241015-e9} and \eqref{equ:241015-e8} that
\begin{align}\label{equ:241015-e10}
	&	\frac{1}{(\mu_{j, \rho}^{(1)})^2}\int_{B_{\mu_{j, \rho}^{(1)} d}(x_{j, \rho}^{(1)})}g_\rho(\frac{x}{\mu_{j, \rho}^{(1)}}+x_{j, \rho}^{(1)})\Phi(x) dx \\ \nonumber
	& =  \frac{1}{(\mu_{j, \rho}^{(1)})^2\rho}  \Big(\int_{B_{\mu_{j, \rho}^{(1)} d}(x_{j, \rho}^{(1)})} \big(\sum_{j=1}^k U_{x_{j, \rho}^{(2)}, \mu_{j, \rho}^{(2)}}(\frac{x}{\mu_{j, \rho}^{(1)}}+x_{j, \rho}^{(1)})\big)\Phi(x)dx  + \int_{B_{\mu_{j, \rho}^{(1)} d}(x_{j, \rho}^{(1)})}  w_{\rho}^{(2)}(\frac{x}{\mu_{j, \rho}^{(1)}}+x_{j, \rho}^{(1)})\Phi(x)dx  \Big) \\ \nonumber
	&\cdot \Big( -\frac{4}{N-2} \int_{\Omega} \big [   \sum_{j=1}^k U_{x_{j, \rho}^{(1)}, \mu_{j, \rho}^{(1)}}(x)  \big]^{\frac{N+2}{N-2}}    \xi_{\rho}   dx  + O(\frac{1}{\bar{\mu}_\rho^{(\frac{3N}{2}-5)}}) \Big) 	\\ \nonumber
	&=  (\mu_{j, \rho}^{(1)})^{(N-2)/2} \int_{B_{\mu_{j, \rho}^{(1)} d}(x_{j, \rho}^{(1)})}   U_{0, 1}(x)\Phi(x)dz \cdot \frac{1}{(\mu_{j, \rho}^{(1)})^2\rho} \Big( -\frac{4}{N-2} \int_{\Omega} \big [   \sum_{j=1}^k U_{x_{j, \rho}^{(1)}, \mu_{j, \rho}^{(1)}}(x)  \big]^{\frac{N+2}{N-2}}    \xi_{\rho}   dx      \Big) +o(\frac{1}{\bar{\mu}_\rho}) \\ \nonumber
		&=\frac{(\mu_{j, \rho}^{(1)})^{(N-2)/2}}{(\mu_{j, \rho}^{(1)})^2\rho} \int_{B_{\mu_{j, \rho}^{(1)} d}(x_{j, \rho}^{(1)})}  U_{0, 1}(x)\Phi(x)dz \cdot \Big(-\frac{4}{N-2} \sum_{j=1}^k 	\frac{1}{(\mu_{j, \rho}^{(1)})^{\frac{N-2}{2}}}\int_{\mathbb{R}^N} \big [    U_{0, 1}(x)  \big]^{\frac{N+2}{N-2}}    \xi_{\rho, j}(x)   dx  \Big) +o(\frac{1}{\bar{\mu}_{\rho}}).
\end{align}	
	
	Thus, we obtain \eqref{equ:241114-e3}.
\end{proof}

	Then from Lemma \ref{lm:241114-l1}, we have the following result:
	\begin{lemma}\label{lm:241114-l2}
		For  $N \geq 6$ and $|\xi_{\rho, j}(x)|\leq 1$, we suppose that   $\xi_{\rho, j}(x) \rightarrow \xi_j(x)$ in $C\left(B_r(0)\right)$. Then  $\xi_j(x)$  satisfies following system:
		\begin{equation*}
			-\Delta \xi_j(x)=\frac{N+2}{N-2} U_{0,1}^{\frac{4}{N-2}}(x) \xi_j(x) -\frac{4 C_*}{N-2}U_{0,1}\int_{\mathbb{R}^N}  U_{0,1}^{\frac{N+2}{N-2}}(x) \xi_j(x)dx ~~~ \text { in } \mathbb{R}^N, \quad \text{for}~~ j=1,...,N,
		\end{equation*}
		where $ C_*  $ is a positive constant. 
	\end{lemma}	
	
	\begin{proof}
		In view of \eqref{equ:241015-e1}, \eqref{equ:241015-e2}, \eqref{equ:241015-e3} and \eqref{equ:241015-e10}, we notice  that 
		\begin{align}\label{equ:241015-e4}
			-\int_{B_{\mu_{j, \rho}^{(1)} d}\left(x_{j, \rho}^{(1)}\right)} \Delta \xi_{\rho, j}(x) \Phi(x) d x=&\left(\frac{N+2}{N-2}\right) \int_{\mathbb{R}^N} U_{0,1}^{\frac{4}{N-2}}(x) \xi_{\rho, j}(x) \Phi(x) d x  +o(\frac{1}{\bar{\mu}_{\rho}}) \\ \nonumber
			&  + \Big(-\frac{4}{N-2} \sum_{j=1}^k 	\frac{1}{(\mu_{j, \rho}^{(1)})^{\frac{N-2}{2}}}\int_{\mathbb{R}^N} \big [    U_{0, 1}(x)  \big]^{\frac{N+2}{N-2}}    \xi_{\rho, j}(x)   dx  \Big)  \\  \nonumber
			& \cdot  \frac{(\mu_{j, \rho}^{(1)})^{(N-2)/2}}{(\mu_{j, \rho}^{(1)})^2\rho} \int_{B_{\mu_{j, \rho}^{(1)} d}(x_{j, \rho}^{(1)})}  U_{0, 1}(x)\Phi(x)dz.
		\end{align}

		Letting $\rho \rightarrow 0$ in \eqref{equ:241015-e4} and using the elliptic regularity theory, we find that $\xi_j(x)$ satisfies
		\begin{equation*}
			-\Delta \xi_j(x)=\frac{N+2}{N-2} U_{0,1}^{\frac{4}{N-2}}(x) \xi_j(x) -\frac{4 C_*}{N-2}U_{0,1}\int_{\mathbb{R}^N}  U_{0,1}^{\frac{N+2}{N-2}}(x) \xi_j(x)dx  \quad \text { in } \mathbb{R}^N,
		\end{equation*}
		where $ C_*  $ is a positive constant. 
	\end{proof}

	Next, we give the kernel of a linearized  operator.
\begin{lemma}\label{lm:241114-l3}
	Let $ \xi_0:=(\xi_1,...,\xi_N)$  be a bounded solution of following system:
	\begin{equation}\label{equ:241015-e11}
		-\Delta \xi_j(x)=\frac{N+2}{N-2} U_{0,1}^{\frac{4}{N-2}}(x) \xi_j(x) -\frac{4 C_*}{N-2}U_{0,1}\int_{\mathbb{R}^N}  U_{0,1}^{\frac{N+2}{N-2}}(x) \xi_j(x)dx  ~~~ \text { in } \mathbb{R}^N, \quad \text{for}~~ j=1,...,N,
	\end{equation}
	where $ C_*  $ is a positive constant. Then it holds 
	\begin{align*}
		\xi_j(x)=\sum\limits_{i=0}^N c_{j, i} \psi_i(x),
	\end{align*}
	where $c_{j, i}, i=0,1, \cdots, N$ are some constants and
	\begin{align*}
			\psi_0(x)=\left.\frac{\partial U_{0, \mu}(x)}{\partial \mu}\right|_{\mu=1}, ~\psi_i(x)=\frac{\partial U_{0,1}(x)}{\partial x_i}, ~i=1, \cdots, N.
	\end{align*}

\end{lemma}

\begin{proof}
 We will show that $ \int_{\mathbb{R}^N}  U_{0,1}^{\frac{N+2}{N-2}}(x) \xi_j(x)dx=0  $.
 Set $$ L(u):= -\Delta u-\frac{N+2}{N-2} U_{0,1}^{\frac{4}{N-2}}(x)u   $$ 
and 
we may choose  
$$ \bar{ \xi }= x\cdot \nabla U_{0,1} + \frac{N-2}{2} U_{0,1}.  $$ 
Then, by direct calculation,  we can deduce that 
\begin{align*}
	L (  \bar{ \xi }  ) =0.
\end{align*}
Moreover,  we  observe  that 
\begin{align}\label{equ:241101-e1}
	\langle 	L (  \xi_j  ) ,\bar{ \xi } \rangle= \langle 	L (  \bar{ \xi }  ) , \xi_j  \rangle=0.
\end{align}
On the other hand, we see that 
\begin{align}\label{equ:241101-e2}
	\langle 	L (  \xi_j  ) ,\bar{ \xi } \rangle = -\frac{4 C_*}{N-2}\int_{\mathbb{R}^N}  U_{0,1}^{\frac{N+2}{N-2}}(x) \xi_j(x)dx  \Big( \int_{\mathbb{R}^N}  U_{0,1} \big (x\cdot \nabla U_{0,1} + \frac{N-2}{2} U_{0,1} \big)dx \Big).
\end{align}
A direct calculation shows that
\begin{align}\label{equ:241101-e3}
	\int_{\mathbb{R}^N}  U_{0,1} \big (x\cdot \nabla U_{0,1} + \frac{N-2}{2} U_{0,1} \big)dx=-	\int_{\mathbb{R}^N}  (U_{0,1})^2 dx \neq 0.
\end{align}
In view of \eqref{equ:241015-e11}, \eqref{equ:241101-e1}, \eqref{equ:241101-e2} and \eqref{equ:241101-e3},  we obtain
\begin{align}\label{equ:241103-e2}
	\int_{\mathbb{R}^N}  U_{0,1}^{\frac{N+2}{N-2}}(x) \xi_j(x)dx=0,
\end{align}
which gives $\xi_j(x)=\sum\limits_{i=0}^N c_{j, i} \psi_i(x)$. 
\end{proof}

	Combing Lemma \ref{lm:241114-l1}, Lemma \ref{lm:241114-l2} with Lemma \ref{lm:241114-l3}, we can obtain the following Lemma.
	\begin{lemma}\label{lm:241103-l1}
		 For $N \geq 6$ and $j=1, \cdots, k$, let $\xi_{\rho, j}(x)=\xi_\rho\big(\frac{x}{\mu_{j, \rho}^{(1)}}+x_{j, \rho}^{(1)}\big)$. Then by taking a subsequence if necessary, we have
		 \begin{align}\label{equ:241015-e12}
			\Big|\xi_{\rho, j}(x)-\sum_{i=0}^N c_{j, i} \psi_i(x)\Big|=o\Big(\frac{1}{\bar{\mu}_\rho}\Big)~ \text{uniformly in} ~ C^1\left(B_R(0)\right) ~ \text{for any}~ R>0,
		\end{align} 
		where $c_{j, i}, i=0,1, \cdots, N$ are some constants and
			
			$$
			\psi_0(x)=\left.\frac{\partial U_{0, \mu}(x)}{\partial \mu}\right|_{\mu=1}, ~\psi_i(x)=\frac{\partial U_{0,1}(x)}{\partial x_i}, ~i=1, \cdots, N.
			$$
		\end{lemma}

	\begin{lemma}
		For $\xi_\rho(x)$ defined by \eqref{equ:241009-e2}, we have
		\begin{align}\label{equ:241008-e1}
			\int_{\Omega} \xi_\rho(x) d x=O\Big(  \frac{ 1 }{    \bar{\mu}_\rho^{N-2} }  \Big)
			\quad \text{and}\quad 
			\xi_\rho(x)=O\Big(  \frac{ 1 }{    \bar{\mu}_\rho^{N-2} }  \Big) \quad  \text{in} ~ \Omega \backslash \bigcup_{j=1}^k B_d(x_{j, \rho}^{(1)}),	
		\end{align}
		where $d>0$ is any small fixed constant.
	\end{lemma}

	\begin{proof}
	By \eqref{equ:241031-e3}, \eqref{equ:241114-e4} and direct calculation, we have 
		\begin{align}\label{equ:241009-e4}
		  	\int_{\Omega} G(y, x)g_{\rho}(y)  dy = & 	\int_{\Omega}  G(y, x)   \frac{\rho \big(\lambda_\rho^{(1)}-\lambda_\rho^{(2)}\big)}{ \rho \| u_\rho^{(1)}- u_\rho^{(2)} \|_{L^{\infty}}} u_\rho^{(2)}(x) dy    \\ \nonumber
			= & O\Big(  \int_{\Omega}   	\frac{1}{  \rho(\mu_{j, \rho}^{(1)})^{\frac{N+2}{2}}}   \frac{1}{ |y-x|^{N-2}}   \big(  \sum_{j=1}^k  U_{x_{j, \rho}^{(2)}, \mu_{j, \rho}^{(2)}}   +w_{\rho}^{(2)} \big)  dy   \Big)   \\ \nonumber                               
			= & O\Big(     \int_{\Omega}  \frac{1}{ \rho(\mu_{j, \rho}^{(1)})^{\frac{N+2}{2}}}   \frac{1}{ |y-x|^{N-2}}   \sum_{j=1}^k     \frac{ \big( \mu_{j, \rho}^{(2)}  \big)^{(N-2)/2}  } {   \big(  1+ |  y-\mu_{j, \rho}^{(2)} x_{j, \rho}^{(2)}  |^2    \big)^{(N-2)/2}      }  dy  \Big)    + O\Big(  \frac{ 1 }{    \bar{\mu}_\rho^N }  \Big)     \\ \nonumber
		= & O\Big(      \sum_{j=1}^k     	\frac{1}{\rho(\mu_{j, \rho}^{(1)})^{\frac{N+2}{2}} }   \int_{\mathbb{R}^N}   \frac{1}{ |y- \mu_{j, \rho}^{(2)}x|^{N-2}}         \frac{ \big( \mu_{j, \rho}^{(2)}  \big)^{(N-6)/2}  } {   \big(  1+ |  y-\mu_{j, \rho}^{(2)} x_{j, \rho}^{(2)}  |^2    \big)^{(N-2)/2}      }  dx   \Big)     + O\Big(  \frac{ 1 }{    \bar{\mu}_\rho^N }  \Big)                                          \\ \nonumber
			=& O \Big( 	\frac{1}{(\mu_{j, \rho}^{(1)})^2} \sum_{j=1}^{k} \frac{1}{\big(  1+ |  y-\mu_{j, \rho}^{(2)} x_{j, \rho}^{(2)}  |^2    \big)^{(N-4)/2} }    \Big)+  O\Big(  \frac{ 1 }{    \bar{\mu}_\rho^N }  \Big)   \\ \nonumber
	=& O \Big( 	\frac{1}{\bar{\mu}_\rho^2} \sum_{j=1}^{k} \frac{1}{\big(  1+ |  y-\bar{\mu}_\rho x_{j, \rho}^{(2)}  |^2    \big)^{(N-4)/2} }    \Big)+  O\Big(  \frac{ 1 }{    \bar{\mu}_\rho^N }  \Big) \\ \nonumber
	=& O\Big(  \frac{ 1 }{    \bar{\mu}_\rho^{N-2} }  \Big).
		\end{align}


	By the potential theory and \eqref{equ:241009-e1}, \eqref{equ:241009-e2} and \eqref{equ:241009-e4}, we have
		\begin{align}\label{equ:241107-e1}
			\xi_\rho(x) &  =\int_{\Omega} G(y, x)\left[C_{\rho}(y)\xi_\rho(y)+\lambda_\rho^{(1)} \xi_\rho(x)+g_{\rho}(y)\right] d y\\ \nonumber
				& =O\bigg(\sum_{j=1}^k \sum_{l=1}^2 \frac{1}{\left(1+\mu _{j, \rho}^{(l)}|x-x_{j, \rho}^{(l)}|\right)^2}\bigg)+O(\lambda_\rho)  +O \Big(  \frac{ 1 }{    \bar{\mu}_\rho^{N-2}  } \Big) \\ \nonumber
			& =O\bigg(\sum_{j=1}^k \sum_{l=1}^2 \frac{1}{\left(1+\mu _{j, \rho}^{(l)}\big|x-x_{j, \rho}^{(l)}\big|\right)^2}\bigg)	+O \Big(  \frac{ 1 }{    \bar{\mu}_\rho^{N-2}  } \Big).
		\end{align}
		
		Next repeating the above process, we know
		\begin{equation*}
				\begin{aligned}
				\xi_\rho(x) & =O\bigg(\int_{\Omega} \frac{1}{|x-y|^{N-2}}(C_{\rho}(y)+\lambda_\rho^{(1)} )\Big(\sum_{j=1}^k \sum_{l=1}^2 \frac{1}{\big(1+\mu _{j, \rho}^{(l)}\big|x-x_{j, \rho}^{(l)}\big|\big)^2}+  \frac{ 1 }{    \bar{\mu}_\rho^{N-2}  }   \Big)\bigg)\\
				&+O\Big(  \frac{ 1 }{    \bar{\mu}_\rho^{N-2} }  \int_{\Omega} \frac{1}{|x-y|^{N-2}} d y\Big) \\
				& =O\left( \sum_{j=1}^k \sum_{l=1}^2 \frac{1}{\big(1+\mu _{j, \rho}^{(l)}|x-x_{j, \rho}^{(l)}|\big)^4}\right)+O\Big(  \frac{ 1 }{    \bar{\mu}_\rho^{N-2} }  \Big).
			\end{aligned}
		\end{equation*}
	Then we can proceed as in the above argument for finite number of times to prove
\begin{align*}
	\xi_\rho(x)=O\Big(\sum_{j=1}^k \sum_{l=1}^2 \frac{\ln \bar{\mu}_\rho}{(1+\mu _{j, \rho}^{(l)}\big|x-x_{j, \rho}^{(l)}\big|)^{N-2}}\Big)+O\Big(  \frac{ 1 }{    \bar{\mu}_\rho^{N-2} }  \Big).
\end{align*}
Hence, we can obtain \eqref{equ:241008-e1}.

	\end{proof}

\begin{proposition}
	 Let $\xi_{\rho}(x)$  be defined as in \eqref{equ:241009-e2}. Then it holds 
	\begin{equation}\label{equ:241017-e6}
		\begin{aligned}
			&\begin{aligned}
				\xi_{\rho}(x)= & \sum_{j=1}^k A_{\rho, j} G\big(x_{j, \rho}^{(1)}, x\big)+\sum_{j=1}^k \sum_{i=1}^N B_{\rho, j, i} \partial_i G\big(x_{j, \rho}^{(1)}, x\big) +O \big(  \frac{ 1 }{    \bar{\mu}_\rho^{N-2}  } \big) ~\text { in } C^1\Big(\Omega \backslash \bigcup_{j=1}^k B_{2 d}(x_{j, \rho}^{(1)})\Big),
			\end{aligned}
		\end{aligned}
	\end{equation}
where $d>0$ is any small fixed constant, $\partial_i G(y, x)=\frac{\partial G(y, x)}{\partial y_i}$,
\begin{align}\label{equ:241016-e1}
	A_{\rho, j}=\int_{B_d\left(x_{j, \rho }^{(1)}\right)} C_{\rho}(x) \xi_{\rho}(x) d x \text { and } B_{\rho, j, i}=\int_{B_d\left(x_{j, \rho}^{(1)}\right)}\left(x_i-x_{j, \rho, i}^{(1)}\right) C_{\rho}(x) \xi_{\rho}(x) d x.
\end{align}

\end{proposition}

\begin{proof}
 By the potential theory and \eqref{equ:241009-e3}, we have
\begin{equation}\label{equ:241016-e3}
	\xi_{\rho}(x)=\int_{\Omega} G(y, x)C_\rho(y) \xi_\rho (y) d y+\lambda_\rho^{(1)}  \int_{\Omega} G(y, x) \xi_{\rho}(y) dy+\int_{\Omega} G(y, x) g_\rho(y)  dy.
\end{equation}

Next, for $x \in \Omega \backslash \bigcup_{j=1}^k B_{2 d}\big(x_{j,\rho}^{(1)}\big)$, by \eqref{equ:241008-e1}, we find 
\begin{equation}\label{equ:241016-e2}
\begin{aligned}
	\int_{\Omega} G(y, x) \xi_{\rho}(y) d y & =\sum_{j=1}^k \int_{B_d\left(x_{j,\rho}^{(1)}\right)} G(y, x) \xi_{\rho}(y) d y+\int_{\Omega \backslash \bigcup_{j=1}^k B_d\left(x_{j, \rho}^{(1)}\right)} G(y, x) \xi_{\rho}(y) d y \\
	& =  O\left(\sum_{j=1}^k \int_{B_d\left(x_{j, \rho}^{(1)}\right)} \xi_{\rho}(y) d y\right) + O\left(\frac{\ln \bar{\mu}_{\rho}}{\bar{\mu}_{\rho}^{N-2}} \int_{\Omega} G(y, x) d y\right)=O\left(\frac{\ln \bar{\mu}_{\rho}}{\bar{\mu}_{\rho}^{N-2}}\right).
\end{aligned}
\end{equation}

Also using \eqref{equ:241008-e1} and \eqref{equ:241015-e12}, we can get	
	\begin{equation}\label{equ:241016-e4}
		\begin{aligned}
		\int_{\Omega} & G(y, x) C_{\rho}(y) \xi_{\rho}(y) d y \\
		= & \sum_{j=1}^k \int_{B_d\left(x_{j, \rho}^{(1)}\right)} G(y, x) C_{\rho}(y) \xi_{\rho}(y) d y+\int_{\Omega \backslash \cup_{j=1}^k B_d\left(x_{j, \rho}^{(1)}\right)} G(y, x) C_{\rho}(y) \xi_{\rho}(y) dy \\
		= & \sum_{j=1}^k A_{\rho, j} G\big(x_{j, \rho}^{(1)}, x\big)+\sum_{j=1}^k \sum_{i=1}^N B_{\rho, j, i} \partial_i G\big(x_{j, \rho}^{(1)}, x\big) \\
		& +O\Big(\sum_{j=1}^k \int_{B_d\left(x_{j, \rho}^{(1)}\right)}\big|y-x_{j, \rho}^{(1)}\big|^2 C_{\rho}(y) \xi_{\rho}(y) d y\Big) +O\Big(\frac{\ln \bar{\mu}_{\rho}}{\bar{\mu}_{\rho}^N} \int_{\Omega \backslash \cup_{j=1}^k B_d\left(x_{j, \rho}^{(1)}\right)} G(y, x) d y\Big) \\
		= & \sum_{j=1}^k A_{\rho, j} G\big(x_{j, \rho}^{(1)}, x\big)+\sum_{j=1}^k \sum_{i=1}^N B_{\rho, j, i} \partial_i G\big(x_{j, \rho}^{(1)}, x\big)+O\Big( \frac{\ln \bar{\mu}_{\rho}}{\bar{\mu}_{\rho}^N} \Big),
	\end{aligned}
\end{equation}
	
where $A_{\rho, j}$ and $B_{\rho, j, i}$ are defined in \eqref{equ:241016-e1}. 
From \eqref{equ:241009-e4} and 	\eqref{equ:241114-e4}, we notice that
\begin{align}\label{equ:241017-e2}
	\int_{\Omega} G(y, x) g_\rho(y)  dy=&   \int_{\Omega} G(y, x)\frac{\lambda_\rho^{(1)}-\lambda_\rho^{(2)}}{\| u_\rho^{(1)}- u_\rho^{(2)} \|_{L^{\infty}}}[\sum_{i=1}^kU_{x_{i, \rho}^{(2)}, \mu_{i, \rho}^{(2)}}(y)+ w_{\rho}^{(2)}(y)]  d y \\ \nonumber
    =&O\Big(  \frac{ 1 }{    \bar{\mu}_\rho^{N-2} }  \Big).
\end{align}

Hence \eqref{equ:241016-e3}, \eqref{equ:241016-e2}, \eqref{equ:241016-e4} and \eqref{equ:241017-e2} imply
$$
\begin{aligned}
	\xi_{\rho}(x)= & \sum_{j=1}^k A_{\rho, j} G\left(x_{j, \rho}^{(1)}, x\right)+\sum_{j=1}^k \sum_{i=1}^N B_{\rho, j, i} \partial_i G\left(x_{j, \rho}^{(1)}, x\right)
	+O \Big(  \frac{ 1 }{    \bar{\mu}_\rho^{N-2}  } \Big)
	 \text { for } x \in \Omega \backslash \bigcup_{j=1}^k B_{2 d}\Big(x_{j, \rho}^{(1)}\Big) .
\end{aligned}
$$
	
On the other hand, from \eqref{equ:241016-e3}, we know

$$
\frac{\partial \xi_{\rho}(x)}{\partial x_i}=\int_{\Omega} D_{x_i} G(y, x) C_{\rho}(y) \xi_{\rho}(y) d y+\rho \int_{\Omega} D_{x_i} G(y, x) \xi_{\rho}(y) d y, \text { for } i=1, \cdots, N.
$$

Then similar to the above estimates of $\xi_{\rho}(x)$, we can complete the proof of \eqref{equ:241017-e6}.	
\end{proof}

\begin{proposition}	
For $N \geq 6$, let $u_{\rho}^{(l)}(x)$ with $l=1,2$ be the solutions of \eqref{equ:original-Problem} satisfying \eqref{1.4}.
Then for small fixed $d>0$, it holds
\begin{equation}\label{equ:241106-e1}
u_{\rho}^{(l)}(x)=A\left(\sum_{j=1}^k \frac{G(x_{j, \rho}^{(1)}, x)}{(\mu_{j, \rho}^{(1)})^{(N-2) / 2}}\right)+O\big(\frac{1}{\bar{\mu}_{\rho}^{(N+2) / 2}}\big) \quad \text { in } ~~ C^1\big(\Omega \backslash \bigcup_{j=1}^k B_{2 d} (x_{j, \rho}^{(1)} )\big)
\end{equation}
where \begin{equation*}
	A=\int_{\mathbb{R}^N} U_{0,1}^{\frac{N+2}{N-2}}, \quad  B=\int_{\mathbb{R}^N} U_{0,1}^2 .
\end{equation*}

\end{proposition}	
	
	\begin{proof}
	The proof is  similar to the one for \cite[Proposition 4.1]{CLP}, so we omit it here.  
\end{proof}
	

\begin{proposition}
	For $\xi_\rho$ defined by \eqref{equ:241009-e2}, we have the following local Pohozaev identities:	
\end{proposition}

\begin{align}\label{equ:241105-e1}
	&	-\int_{\partial \Omega^{\prime}} \frac{\partial \xi_{\rho}}{\partial \nu} \frac{\partial u_{\rho}^{(1)}}{\partial x_i}-\int_{\partial \Omega^{\prime}} \frac{\partial u_{\rho}^{(2)}}{\partial \nu} \frac{\partial \xi_{\rho}}{\partial x_i}+\frac{1}{2} \int_{\partial \Omega^{\prime}}\left\langle\nabla\left(u_{\rho}^{(1)}+u_{\rho}^{(2)}\right), \nabla \xi_{\rho}\right\rangle \nu_i \\ \nonumber
	&	=\frac{N-2}{2 N} \int_{\partial \Omega^{\prime}} D_{\rho}(x) \xi_{\rho} \nu_i+  \frac{ \lambda^{(1)}_{\rho}  }{2} \int_{\partial \Omega^{\prime}}\left(u_{\rho}^{(1)}+u_{\rho}^{(2)}\right) \xi_{\rho} \nu_i+ 	\frac{  (\lambda_\rho^{(1)}-\lambda_\rho^{(2)}  )}{2\| u_\rho^{(1)}- u_\rho^{(2)} \|_{L^{\infty}}}  
		\int_{\partial \Omega^{\prime}}\left(u_{\rho}^{(1)}\right)^2 \cdot \nu_i,
\end{align}
and
\begin{equation}\label{equ:241105-e2}
	\begin{aligned}
		& \frac{1}{2} \int_{\partial \Omega^{\prime}}\left\langle\nabla\left(u_{\rho}^{(1)}+u_{\rho}^{(2)}\right), \nabla \xi_{\rho}\right\rangle\left\langle x-x_{j, \rho}^{(1)}, \nu\right\rangle-\int_{\partial \Omega^{\prime}} \frac{\partial \xi_{\rho}}{\partial \nu}\left\langle x-x_{j, \rho}^{(1)}, \nabla u_{\rho}^{(1)}\right\rangle \\
		& \quad-\int_{\partial \Omega^{\prime}} \frac{\partial u_{\rho}^{(2)}}{\partial \nu}\left\langle x-x_{j, \rho}^{(1)}, \nabla \xi_{\rho}\right\rangle+\frac{2-N}{2} \int_{\partial \Omega^{\prime}} \frac{\partial \xi_{\rho}}{\partial \nu} u_{\rho}^{(1)}+\frac{2-N}{2} \int_{\partial \Omega^{\prime}} \frac{\partial u_{\rho}^{(2)}}{\partial \nu} \xi_{\rho} \\
		& =\int_{\partial \Omega^{\prime}} D_{\rho}(x) \xi_{\rho}\left\langle x-x_{j, \rho}^{(1)}, \nu\right\rangle+      \frac{ \lambda^{(2)}_{\rho}   }{2} \int_{\partial \Omega^{\prime}}\left(u_{\rho}^{(1)}+u_{\rho}^{(2)}\right) \xi_{\rho}\left\langle x-x_{j, \rho}^{(1)}, \nu\right\rangle -\lambda^{(2)}_{\rho} \int_{\Omega^{\prime}}\left(u_{\rho}^{(1)}+u_{\rho}^{(2)}\right) \xi_{\rho}\\
		& -	\frac{  (\lambda_\rho^{(1)}-\lambda_\rho^{(2)}  )}{2\| u_\rho^{(1)}- u_\rho^{(2)} \|_{L^{\infty}}}  
		\int_{\partial \Omega^{\prime}}\left(u_{\rho}^{(1)}\right)^2 \xi_{\rho}\left\langle x-x_{j, \rho}^{(1)}, \nu\right\rangle - 	\frac{  (\lambda_\rho^{(1)}-\lambda_\rho^{(2)}  )}{\| u_\rho^{(1)}- u_\rho^{(2)} \|_{L^{\infty}}}  
		\int_{ \Omega^{\prime}}\left(u_{\rho}^{(1)}\right)^2.
	\end{aligned}
\end{equation}		
where $\Omega^{\prime} \subset \Omega$ is a smooth domain, $\nu(x)=\left(\nu_1(x), \cdots, \nu_N(x)\right)$ is the outward unit normal of $\partial \Omega^{\prime}$ and
$$
D_{\rho}(x)=\int_0^1\left(t u_{\rho}^{(1)}(x)+(1-t) u_{\rho}^{(2)}(x)\right)^{\frac{N+2}{N-2}} dt.
$$
\begin{proof}
	Taking $u_\rho = u^{(l)}_\rho$   with $l=1,2  $ in \eqref{equ:241105-e3} and \eqref{equ:241105-e4}, and then making a difference
	between those respectively, we can obtain \eqref{equ:241105-e1} and \eqref{equ:241105-e2}. 
\end{proof}

\begin{proposition}
	For $N \geq 6$, it holds
	\begin{align}\label{equ:241107-e2}
		c_ {j,0}=0, \quad \text{for}~j=1,...,k,
	\end{align}
	where $c_ {j,0}$ are the constants in Lemma \ref{lm:241103-l1}.
\end{proposition}

\begin{proof}

 First, we define the following quadric form
\begin{equation*}
	\begin{aligned}
		P_1(u, v)= & -\theta \int_{\partial B_\theta\left(x_{j, \rho}^{(1)}\right)}\langle\nabla u, \nu\rangle\langle\nabla v, \nu\rangle+\frac{\theta}{2} \int_{\partial B_\theta\left(x_{j, \rho}^{(1)}\right)}\langle\nabla u, \nabla v\rangle \\
		& +\frac{2-N}{4} \int_{\partial B_\theta\left(x_{j, \rho}^{(1)}\right)}\langle\nabla u, \nu\rangle v+\frac{2-N}{4} \int_{\partial B_\theta\left(x_{j, \rho}^{(1)}\right)}\langle\nabla v, \nu\rangle u .
	\end{aligned}
\end{equation*}

Note that if $u$ and $v$ are harmonic in $B_d(x_{j, \rho}^{(1)}) \backslash\{x_{j, \rho}^{(1)}\}$, then $P_1(u, v)$ is independent of $\theta \in(0, d]$. So using \eqref{equ:241017-e6} and \eqref{equ:241106-e1}, we get by taking $\Omega^{\prime}=B_\theta(x_{j, \rho}^{(1)})$ in \eqref{equ:241105-e2}, for $N \geq 6$,
\begin{equation}\label{equ:241106-e6}
\begin{aligned}
		\text { LHS of \eqref{equ:241105-e2} }= & \sum_{l=1}^k \sum_{m=1}^k \frac{2 A A_{\rho, l}}{(\mu_{m, \rho}^{(1)})^{(N-2) / 2}} P_1\left(G(x_{m, \rho}^{(1)}, x), G(x_{l, \rho}^{(1)}, x)\right) \\
		& +\sum_{l=1}^k \sum_{m=1}^k \sum_{h=1}^N \frac{2 A B_{\rho, l, h}}{(\mu_{m, \rho}^{(1)})^{(N-2) / 2}} P_1\left(G(x_{m, \rho}^{(1)}, x), \partial_h G(x_{l, \rho}^{(1)}, x)\right) \\
		& +O\big(\frac{1}{\bar{\mu}_{\rho}^{(3 N-2) / 2}}\big) .
\end{aligned}
\end{equation}
  A very similar calculation progress as the Section 5 in \cite{CLP}, can be used to show \eqref{equ:241106-e2} and \eqref{equ:241106-e3}. So we omit it here.
Thus, from Section 5 in \cite{CLP}, we have the following estimate:
\begin{equation}\label{equ:241106-e2}
	P_1\left(G(x_{m, \rho}^{(1)}, x), G(x_{l, \rho}^{(1)}, x)\right)= \begin{cases}-\frac{(N-2) R(x_{j, \rho}^{(1)})}{2}, & \text { for } l, m=j .  \vspace{3mm} \\  \vspace{3mm}
		\frac{(N-2) G(x_{j, \rho}^{(1)}, x_{l, \rho}^{(1)})}{4}, & \text { for } m=j, l \neq j . \\ \vspace{3mm}
		\frac{(N-2) G(x_{j, \rho}^{(1)}, x_{m, \rho}^{(1)})}{4}, & \text { for } m \neq j, l=j . \\  
		0, & \text { for } l, m \neq j .\end{cases}
\end{equation}
	and
\begin{equation}\label{equ:241106-e3}
	P_1\left(G(x_{m, \rho}^{(1)}, x), \partial_h G(x_{l, \rho}^{(1)}, x)\right)= \begin{cases}\left(\frac{N-2}{4}+\frac{1-N}{2 N}\right) \partial_h R(x_{j, \rho}^{(1)}), & \text { for } l, m=j . \vspace{3mm} \\  \vspace{3mm}
		 \frac{(N-2)}{4} \partial_h G(x_{l, \rho}^{(1)}, x_{j, \rho}^{(1)}), & \text { for } m=j, l \neq j . \\ \vspace{3mm}
		 
		  \left(\frac{N-2}{4}+\frac{1-N}{N}\right) \partial_h G(x_{j, \rho}^{(1)}, x_{m, \rho}^{(1)}), & \text { for } m \neq j, l=j . \\ 0, & \text { for } l, m \neq j .\end{cases}
\end{equation}
	
	On the other hand, from \eqref{equ:241008-e1},\eqref{equ:241015-e12},\eqref{equ:241106-e4},\eqref{equ:241016-e1} and 	\eqref{equ:241114-e4}, we deduce 
	\begin{align}\label{equ:241106-e5}
		A_{\rho, j}=\left(\frac{N+2}{N-2}\right) \int_{B_d\left(x_{j, \rho}^{(1)}\right)} U_{x_{j, \rho}^{(1)}, \mu_{j, \rho}^{(1)}}^{\frac{4}{N-2}} \xi_{\rho}+o(\frac{1}{\bar{\mu}_{\rho}^{N-1}})=-\frac{(N-2) A c_{j, 0}}{2(\mu_{j, \rho}^{(1)})^{N-2}}+o(\frac{1}{\bar{\mu}_{\rho}^{N-1}}) .
\end{align}

	Now let $d_{j, \rho}=\frac{(N-2) A^2 c_{j, 0}}{8\left(\mu_{j, \rho}^{(1)}\right)^{(N-2) / 2}}$, for $j=1, \cdots, k$. Then \eqref{equ:241106-e6}-\eqref{equ:241106-e5} imply
		\begin{align} \label{equ:241106-e9}
			\text {LHS of \eqref{equ:241105-e2} }= & 2(N-2) d_{j, \rho}\Big(\frac{R(x_{j, \rho}^{(1)})}{(\mu_{j, \rho}^{(1)})^{N-2}}-\sum_{l \neq j}^k \frac{G(x_{j, \rho}^{(1)}, x_{l, \rho}^{(1)})}{(\mu_{j, \rho}^{(1)})^{(N-2) / 2}(\mu_{l, \rho}^{(1)})^{(N-2) / 2}}\Big) \\ \nonumber
			& +2(N-2)\Big(\frac{d_{j, \rho} R(x_{j, \rho}^{(1)})}{(\mu_{j, \rho}^{(1)})^{N-2}}-\sum_{l \neq j}^k \frac{d_{l, \rho} G(x_{j, \rho}^{(1)}, x_{l, \rho}^{(1)})}{(\mu_{j, \rho}^{(1)})^{(N-2) / 2}(\mu_{l, \rho}^{(1)})^{(N-2) / 2}}\Big) \\ \nonumber
			& +\frac{N-2}{2} \sum_{h=1}^N B_{\rho, j, h} \Big(\frac{\partial_h R (x_{j, \rho}^{(1)})}{(\mu_{j, \rho}^{(1)})^{(N-2) / 2}}-\sum_{l \neq j}^k \frac{\partial_h G (x_{j, \rho}^{(1)}, x_{l, \rho}^{(1)})}{  (\mu_{l, \rho}^{(1)})^{ (N-2) / 2}      }  \Big) \\ \nonumber
			& -\frac{N-2}{2} \sum_{h=1}^N \sum_{l \neq j}^k \frac{B_{\rho, l, h}}{(\mu_{j, \rho}^{(1)})^{(N-2) / 2}} \partial_h G(x_{j, \rho}^{(1)}, x_{l, \rho}^{(1)})+o \big (\frac{1}{\bar{\mu}_\rho^{(3 N-4) / 2}}  \big) \\ \nonumber
			& +\frac{(N-1)}{N} \sum_{h=1}^N B_{\rho, j, h}\Big(\frac{\partial_h R(x_{j, \rho}^{(1)})}{(\mu_{j, \rho}^{(1)})^{N-2}}-2 \sum_{l \neq j}^k \frac{\partial_h G\left(x_{j, \rho}^{(1)}, x_{l, \rho}^{(1)}\right)}{(\mu_{j, \rho}^{(1)})^{(N-2) / 2}(\mu_{l, \rho}^{(1)})^{(N-2) / 2}}\Big) .
		\end{align}	
		
		Moreover, in view of  \cite[Proposition 2.5]{CLP},  we can notice that 
		\begin{align}\label{equ:241106-e7}
			\frac{R(x_{j, \rho}^{(1)})}{(\mu_{j, \rho}^{(1)})^{N-2}}-\sum_{l \neq j}^k \frac{G(x_{j, \rho}^{(1)}, x_{l, \rho}^{(1)})}{(\mu_{j, \rho}^{(1)})^{(N-2) / 2}(\mu_{l, \rho}^{(1)})^{(N-2) / 2}} \\ \nonumber
			& =\frac{2 B \lambda_\rho}{A^2(N-2) (\mu^{(1)}_{j, \rho})^2}+O\Big(\frac{1}{\bar{\mu}_{\rho}^N}+\frac{\lambda_\rho}{\bar{\mu}_{\rho}^{N-2}}\Big)+o\Big(\frac{\lambda_\rho}{\bar{\mu}_{\rho}^{(N+2) / 2}}+\frac{\lambda_\rho^2}{\bar{\mu}_{\rho}^4}\Big).
		\end{align}

In order to obtain the right hand of \eqref{equ:241105-e2}, we shall do the following estimates:

By  Lemma \ref{lm:241103-l1} and formulas \eqref{equ:241103-e2}, we  derive  that 
\begin{align}\label{equ:241106-e14}
	\frac{ \rho (\lambda_\rho^{(1)}-\lambda_\rho^{(2)}  )}{\| u_\rho^{(1)}- u_\rho^{(2)} \|_{L^{\infty}}}
	& = -\frac{4}{N-2} 	\frac{1}{(\mu_{ \rho}^{(1)})^{\frac{N-2}{2}}}  \int_{\mathbb{R}^N} \big [    U_{0, 1}(x)  \big]^{\frac{N+2}{N-2}}    \hat{\xi}_{\rho}(x)   dx +O(\frac{1}{\bar{\mu}_\rho^{(\frac{3N}{2}-5)}}) \\ \nonumber
	&=-\frac{4}{N-2} 	\frac{1}{(\mu_{ \rho}^{(1)})^{\frac{N-2}{2}}}  \int_{\mathbb{R}^N} \big [    U_{0, 1}(x)  \big]^{\frac{N+2}{N-2}}    \sum_{i=0}^N c_i \psi_i(x)  dx+o(\frac{1}{\bar{\mu}^{\frac{N }{2}}_\rho}) +O(\frac{1}{\bar{\mu}_\rho^{(\frac{3N}{2}-5)}})\\ \nonumber
	&=o(\frac{1}{\bar{\mu}^{\frac{N }{2}}_\rho}) +O(\frac{1}{\bar{\mu}_\rho^{(\frac{3N}{2}-5)}}).
\end{align}	

A direct computation shows the validity of the following:

In view of \eqref{equ:241106-e1} and \eqref{equ:241106-e14}, we have 
	\begin{align*}
			-	\frac{  (\lambda_\rho^{(1)}-\lambda_\rho^{(2)}  )}{2\| u_\rho^{(1)}- u_\rho^{(2)} \|_{L^{\infty}}}  
			\int_{\partial \Omega^{\prime}}\left(u_{\rho}^{(1)}\right)^2 \xi_{\rho}\left\langle x-x_{j, \rho}^{(1)}, \nu\right\rangle &= o(\frac{1}{\bar{\mu}^{N-1}_\rho}),
		\end{align*}		
and 	
			\begin{align*}
			- 	\frac{  (\lambda_\rho^{(1)}-\lambda_\rho^{(2)}  )}{\| u_\rho^{(1)}- u_\rho^{(2)} \|_{L^{\infty}}}  
			\int_{ \Omega^{\prime}}\left(u_{\rho}^{(1)}\right)^2&=o(\frac{1}{\bar{\mu}^{\frac{N }{2}}_\rho}) +O(\frac{1}{\bar{\mu}_\rho^{(\frac{3N}{2}-5)}}).
		\end{align*}
Also, from \eqref{equ:241106-e7}, \eqref{equ:241008-e1}, \eqref{equ:241031-e2} and 	\eqref{equ:241114-e4}, we know 
   \begin{align*}
	  	\text { RHS  of \eqref{equ:241105-e2} } & =-2 \lambda^{(2)}_\rho \int_{B_\theta\left(x_{j, \rho}^{(1)}\right)} U_{x_{j, \rho}^{(1)}, \mu_{j, \rho}^{(1)}}(x) \xi_{\rho}+o\Big(\frac{1}{\bar{\mu}_{\rho}^{(3 N-2) / 2}}\Big) +o(\frac{1}{\bar{\mu}^{\frac{N }{2}}_\rho}) +O(\frac{1}{\bar{\mu}_\rho^{(\frac{3N}{2}-5)}}) +o(\frac{1}{\bar{\mu}^{N-1}_\rho}) \\ \nonumber
		& =\frac{2 B c_{j, 0} \rho}{(\mu_{j, \rho}^{(1)})^{(N+2) / 2}} +o(\frac{1}{\bar{\mu}^{N-1}_\rho})  +o(\frac{1}{\bar{\mu}^{\frac{N }{2}}_\rho}) \\ \nonumber
		& =8 d_{j, \rho}\Big(\frac{R(x_{j, \rho}^{(1)})}{(\mu_{j, \rho}^{(1)})^{N-2}}-\sum_{l \neq j}^k \frac{G (x_{j, \rho}^{(1)}, x_{l, \rho}^{(1)})}{\big(\mu_{j, \rho}^{(1)}\big)^{(N-2) / 2}(\mu_{l, \rho}^{(1)})^{(N-2) / 2}}\Big)+o(\frac{1}{\bar{\mu}^{N-1}_\rho}) +o(\frac{1}{\bar{\mu}^{\frac{N }{2}}_\rho}).
	\end{align*}

	From \cite[Proposition 2.5]{CLP},  if $N \geq 6$, we have
	\begin{equation}\label{equ:241106-e8}
		\frac{1}{2 (\mu_{j, \rho}^{(1)})^{N-2}} \frac{\partial R(x_{j, \rho})}{\partial x_i}-\sum_{l=1, l \neq j}^k \frac{1}{(\mu_{j, \rho}^{(1)})^{(N-2) / 2} (\mu_{j, \rho}^{(1)})^{(N-2) / 2}} \frac{\partial G (x_{j, \rho}, x_{l, \rho})}{\partial x_i}= O(\frac{1}{\bar{\mu}_{\rho}^N})
	\end{equation}
	It follows from \eqref{equ:241106-e8} that 
              \begin{align}\label{equ:241106-e10}
				& \sum_{h=1}^k B_{\rho, j, h}\Big(\frac{\partial_h R(x_{j, \rho}^{(1)})}{(\mu_{j, \rho}^{(1)})^{N-2}}-2 \sum_{l \neq j}^k \frac{\partial_h G(x_{j, \rho}^{(1)}, x_{l, \rho}^{(1)})}{(\mu_{j, \rho}^{(1)})^{(N-2) / 2}(\mu_{l, \rho}^{(1)})^{(N-2) / 2}}\Big) \\ \nonumber
				& \quad=O\Big(\sum_{h=1}^k \frac{B_{\rho, j, h}}{(\mu_{j, \rho}^{(1)})^{N / 2}} \big(\partial_h \Psi_k (a^k, \mu^k)+o(1)\big)\Big)=o\Big(\frac{1}{\bar{\mu}_{\rho}^{(3 N-2) / 2}}\Big) .
			\end{align}

	Hence, \eqref{equ:241106-e6} and \eqref{equ:241106-e9}-\eqref{equ:241106-e10} imply, for  $j=1, \cdots, k$,
          	\begin{align}\label{equ:241106-e12}
				& (N-6) d_{j, \rho}\Big(\frac{R (x_{j, \rho}^{(1)})}{(\mu_{j, \rho}^{(1)} )^{N-2}}-\sum_{l \neq j}^k \frac{G(x_{j, \rho}^{(1)}, x_{l, \rho}^{(1)})}{(\mu_{j, \rho}^{(1)})^{(N-2) / 2}(\mu_{l, \rho}^{(1)})^{(N-2) / 2}}\Big) \\ \nonumber
				& +(N-2)\Big(\frac{d_{j, \rho} R(x_{j, \rho}^{(1)})}{(\mu_{j, \rho}^{(1)})^{N-2}}-\sum_{l \neq j}^k \frac{d_{l, \rho} G(x_{j, \rho}^{(1)}, x_{l, \rho}^{(1)})}{(\mu_{j, \rho}^{(1)})^{(N-2) / 2}(\mu_{l, \rho}^{(1)})^{(N-2) / 2}}\Big) \\ \nonumber
				& +\frac{N-2}{4} \sum_{h=1}^N B_{\rho, j, h}\Big(\frac{\partial_h R(x_{j, \rho}^{(1)})}{ (\mu_{j, \rho}^{(1)})^{(N-2) / 2}}-\sum_{l \neq j}^k \frac{\partial_h G(x_{j, \rho}^{(1)}, x_{l, \rho}^{(1)})}{(\mu_{l, \rho}^{(1)})^{(N-2) / 2}} \Big)\\ \nonumber
				& -\frac{N-2}{4} \sum_{h=1}^N \sum_{l \neq j}^k \frac{B_{\rho, l, h}}{(\mu_{j, \rho}^{(1)})^{(N-2) / 2}} \partial_h G\left(x_{j, \rho}^{(1)}, x_{l, \rho}^{(1)}\right) +o(\frac{1}{\bar{\mu}^{\frac{N }{2}}_\rho}) =o(\frac{1}{\bar{\mu}_{\rho}^{(N-1)}}).
			\end{align}
	Let $a^k=\left(a_1, \cdots, a_k\right) \in \Omega^k$ and $a_j=\left(y_{(j-1) N+1}, y_{(j-1) N+2}, \cdots, y_{j N}\right) \in \Omega$. Since for any $i=\{1, \cdots, k N\}$, there exists some $j \in\{1, \cdots, k\}$ satisfying $i \in[(j-1) N+$ $1, j N] \cap \mathbb{N}^{+}$. Then by direct calculation, we have
	\begin{align*}
		\frac{\partial^2 \Psi_k(x, \mu)}{\partial y_i \partial \mu_m} 	= & (N-2)\Big(\mu_j^{N-3} \frac{\partial R\left(x_j\right)}{\partial y_i}-\sum_{l \neq j}^k \mu_j^{\frac{N-4}{2}} \mu_l^{\frac{N-2}{2}} \frac{\partial G\left(x_j, x_l\right)}{\partial y_i}\Big), \text { if } m \in[(j-1) N+1, j N] \bigcap \mathbb{N}^{+},
	\end{align*}
	and
	\begin{align*}
	\frac{\partial^2 \Psi_k(x, \mu)}{\partial y_i \partial \mu_m}=-(N-2) \mu_j^{\frac{N-2}{2}} \mu_s^{\frac{N-4}{2}} \frac{\partial G\left(x_j, x_s\right)}{\partial y_i} \text {, if } m \in[(s-1) N+1, s N] \bigcap \mathbb{N}^{+} \text {and } s \neq j \text {. }
	\end{align*}	
	Now we rewrite \eqref{equ:241106-e12} as follows:
	\begin{align}\label{equ:241106-e13}
	\bar{M}_{k, \rho} \vec{D}_k+\frac{2}{A^2} \widetilde{D}_k  \Big[  \left(D_{\mu, x}^2 \Psi_k(x, \mu)\right)_{(x, \mu)  =\left(a^k, \mu^k \right)}   +o(\frac{1}{\bar{\mu}_\rho}) \Big]   \widetilde{B}_{\rho, k}=\Big(o(\frac{1}{\bar{\mu}_{\rho}^{N-1}}), \cdots, o(\frac{1}{\bar{\mu}_{\rho}^{N-1}})\Big)^T,
	\end{align}	
		with the vector $\vec{D}_k$, the matrix $\bar{M}_{k, \rho}=\left(a_{i, j, \rho}\right)_{1 \leq i, j \leq k}$ defined by
	\begin{align*}
		\vec{D}_k=\left(c_{1,0}, \cdots, c_{k, 0}\right)^T, \widetilde{D}_k=\operatorname{diag}\left(\mu_1^{\frac{4-N}{2}}, \cdots, \mu_k^{\frac{4-N}{2}}\right),  ~\text{ with} \quad  \mu_j:=\lim\limits _{\rho \rightarrow 0}(\lambda_\rho^{\frac{1}{N-4}} \mu_{j, \rho})^{-1},
	\end{align*}
	
	\begin{equation*}
		a_{i, j,\rho}= \begin{cases}\frac{2(N-4) R\left(x_{j,\rho}^{(1)}\right)}{\left(\mu_{j,\rho}^{(1)}\right)^{N-2}}-\sum_{l \neq j}^k \frac{(N-6) G\left(x_{j,\rho}^{(1)}, x_{l,\rho}^{(1)}\right)}{\left(\mu_{j,\rho}^{(1)}\right)^{(N-2) / 2}\left(\mu_{l,\rho}^{(1)}\right)^{(N-2) / 2}}, & \text { for } i=j, \\ -\frac{(N-2) G\left(x_{j,\rho}^{(1)}, x_{i,\rho}^{(1)}\right)}{\left(\mu_{j,\rho}^{(1)}\right)^{(N-2) / 2}\left(\mu_{i,\rho}^{(1)}\right)^{(N-2) / 2},} & \text { for } i \neq j,\end{cases}
	\end{equation*}
	and
	\begin{align}\label{equ:241106-e15}
	\widetilde{B}_{\rho, k}=\left(\bar{B}_{\rho, 1}, \cdots, \bar{B}_{\rho, N k}\right)^T, \text { with }~ \bar{B}_{\rho, m}=B_{\rho, j, h},~ m=h+N(j-1) .
\end{align}

	Since $\bar{M}_{k, \rho}$ is the main diagonally dominant matrix if $N \geq 6$, we see that $\bar{M}_{k, \rho}$ is invertible. So \eqref{equ:241106-e13} means \eqref{equ:241107-e2}. Then from \eqref{equ:241106-e5} and \eqref{equ:241106-e13}, we obtain
	\begin{align}\label{equ:241106-e16}
		A_{\rho, j}=o\big(\frac{1}{\bar{\mu}_{\rho}^{N-1}}\big) \quad \text{for} \quad j=1, \cdots, k \quad \text{and} \quad N \geq 6.
	\end{align}
	Therefore we find 
	\begin{align}\label{equ:241106-e21}
		&\left(D_{\mu, x}^2 \Psi_k(x, \mu)\right)_{(x, \mu)=\big(a^k, \mu^k \big)} \widetilde{B}_{\rho, k}=\Big (o(\frac{1}{\bar{\mu}_{\rho}^{N-1}}), \cdots, o(\frac{1}{\bar{\mu}_{\rho}^{N-1}})\Big)^T .
	\end{align}

\end{proof}

\begin{proposition}
	For $N \geq 6$, it holds
	\begin{align}\label{equ:241106-e26}
		c_ {j,i}=0, \quad \text{for}~j=1,...,k,~ \text{and} \quad  i = 1,... ,N,
	\end{align}
	where $c_ {j,i}$ are the constants in Lemma \ref{lm:241103-l1}.
\end{proposition} 
	
\begin{proof}
 First, we may define the following quadratic form
\begin{align*}
    Q_1(u, v)=-\int_{\partial B_\theta\left(x_{j, \rho}\right)} \frac{\partial v}{\partial \nu} \frac{\partial u}{\partial x_i}-\int_{\partial B_\theta\left(x_{j, \rho}\right)} \frac{\partial u}{\partial \nu} \frac{\partial v}{\partial x_i}+\int_{\partial B_\theta\left(x_{j, \rho}\right)}\langle\nabla u, \nabla v\rangle \nu_i .
\end{align*}
Next,  taking $ \Omega^{\prime}=B_\theta(x_{j,\lambda}^{(1)}) $ in \eqref{equ:241105-e1}, from \eqref{equ:241017-e6} and \eqref{equ:241106-e1}, we have 
			\begin{align}\label{equ:241106-e17}
				& \text { LHS of \eqref{equ:241105-e1} }=\sum_{l=1}^k \sum_{m=1}^k \frac{A A_{\rho, l} Q_1\left(G(x_{m, \rho}^{(1)}, x), G(x_{l, \rho}^{(1)}, x)\right)}{(\mu_{m, \rho}^{(1)})^{(N-2) / 2}} \\ \nonumber
				& \quad+\sum_{l=1}^k \sum_{h=1}^N \sum_{m=1}^k \frac{A B_{\rho, l, h} Q_1\left(G(x_{m, \rho}^{(1)}, x), \partial_h G(x_{l, \rho}^{(1)}, x)\right)}{(\mu_{m, \rho}^{(1)})^{(N-2) / 2}}+O\big (\frac{\ln \bar{\mu}_{\rho}}{\bar{\mu}_{\rho}^{(3 N-2) / 2}}\big),
			\end{align}

In order to obtain the right hand of \eqref{equ:241105-e1}, we shall do the following estimates:
In view of \eqref{equ:241106-e1} and \eqref{equ:241106-e14}, we have 
\begin{align*}
	\frac{  (\lambda_\rho^{(1)}-\lambda_\rho^{(2)}  )}{2\| u_\rho^{(1)}- u_\rho^{(2)} \|_{L^{\infty}}}  
	\int_{\partial \Omega^{\prime}}\left(u_{\rho}^{(1)}\right)^2 \cdot \nu_i=O(\frac{1}{\bar{\mu}_{\rho}^{ (\frac{3N}{2}-5)  }})+o(\frac{1}{\bar{\mu}^{N-1}_\rho}).
\end{align*}
Thus, we can observe that
	\begin{equation}\label{equ:241106-e18}
		\text { RHS of \eqref{equ:241105-e1} }=O(\frac{1}{\bar{\mu}_{\rho}^{ ( 3N-2)/2  }})+O(\frac{1}{\bar{\mu}_{\rho}^{ (\frac{3N}{2}-5)  }})+o(\frac{1}{\bar{\mu}^{N-1}_\rho})=o(\frac{1}{\bar{\mu}^{N-1}_\rho}).
	\end{equation}
		 Then from \eqref{equ:241106-e16}, \eqref{equ:241106-e17} and \eqref{equ:241106-e18}, we find 
        \begin{align*}
			&\sum_{l=1}^k \sum_{h=1}^N \sum_{m=1}^k \frac{B_{\rho, l, h} Q_1 \big(G(x_{m, \rho}^{(1)}, x), \partial_h G(x_{l, \rho}^{(1)}, x) \big)}{(\mu_{m, \rho}^{(1)})^{(N-2) / 2}}=o(\frac{1}{\bar{\mu}^{N-1}_\rho}).
		\end{align*}

	A very similar calculation progress as the Section 5 in \cite{CLP}, can be used to show \eqref{equ:241106-e19}. So we omit it here.
	Thus, we also have the following estimate:
	 \begin{align}\label{equ:241106-e19}
	 	 Q_1\left(G\left(x_{m, \rho}^{(1)}, x\right), \partial_h G (x_{l, \rho}^{(1)}, x)\right)= \begin{cases}-\partial_{x_i x_h}^2 R\big(x_{j, \rho}^{(1)}\big), & \text { for } l, m=j,  \vspace{2mm} \\  \vspace{2mm}
	 	 	D_{x_i x_h}^2 G\big(x_{m, \rho}^{(1)}, x_{j, \rho}^{(1)}\big), & \text { for } m \neq j, l=j, \\ \vspace{2mm}
	 	 	D_{x_i} \partial_h G\big(x_{l, \rho}^{(1)}, x_{j, \rho}^{(1)}\big), & \text { for } m=j, l \neq j, \\  
	 	 	0, & \text { for } l, m \neq j .\end{cases}
	 \end{align}
	
	 Then \eqref{equ:241106-e18} and \eqref{equ:241106-e19} imply 
    \begin{align}\label{equ:241106-e20}
				\sum_{h=1}^N B_{\rho, j, h} & \left(\frac{\partial_{x_i x_h}^2 R(x_{j, \rho}^{(1)})}{(\mu_{j, \rho}^{(1)})^{(N-2) / 2}}-\sum_{l \neq j} \frac{\partial_{x_i x_h}^2 G(x_{j, \rho}^{(1)}, x_{l, \rho}^{(1)})}{(\mu_{l, \rho}^{(1)})^{(N-2) / 2}}\right) \\ \nonumber
				& \quad-\sum_{h=1}^N \sum_{m \neq j} \frac{B_{\rho, m, h} D_{x_h} \partial_{x_i} G(x_{j, \rho}^{(1)}, x_{m, \rho}^{(1)})}{(\mu_{j, \rho}^{(1)})^{(N-2) / 2}}=o(\frac{1}{\bar{\mu}^{N-1}_\rho}) .
		\end{align}
Since for any $i=\{1, \cdots, k N\}$, there exists some $j \in\{1, \cdots, k\}$ satisfying $i \in [(j-1) N+1, j N]$ $\cap \mathbb{N}^{+}$, by direct calculations, we have
	\begin{align*}
		\frac{\partial^2 \Psi_k(x, \mu)}{\partial y_i \partial y_m}=\mu_j^{N-2} \frac{\partial^2 R\left(x_j\right)}{\partial y_i \partial y_m}-\sum_{l \neq j}^k \mu_j^{\frac{N-2}{2}} \mu_l^{\frac{N-2}{2}} & \frac{\partial^2 G\left(x_j, x_l\right)}{\partial y_i \partial y_m}, \\ 
		& \text { if } m \in[(j-1) N+1, j N] \cap \mathbb{N}^{+},
	\end{align*}
	and
	\begin{align*}
	\frac{\partial^2 \Psi_k(x, \mu)}{\partial y_i \partial y_m}=-\mu_j^{\frac{N-2}{2}} \mu_s^{\frac{N-2}{2}} \frac{\partial^2 G\left(x_j, x_s\right)}{\partial y_i \partial y_m}, ~\text{if}~ m \in[(s-1) N+1, s N] \cap \mathbb{N}^{+}~\text{and}~   s \neq j.
		\end{align*}
	
	So from \eqref{equ:241106-e20}, we can obtain
	\begin{align}\label{equ:241106-e22}
			\left(D_{x x}^2 \Psi_k(x, \mu)\right)_{(x, \mu)=\left(a^k,\mu^k \right)} \widetilde{B}_{\rho, k}=\Big(o(\frac{1}{\bar{\mu}_{\rho}^{N-1}}), \cdots, o(\frac{1}{\bar{\mu}_{\rho}^{N-1}})\Big)^T,
	\end{align}
	where $\widetilde{B}_{\rho, k}$ is the vector in \eqref{equ:241106-e15}. Noting that $\left(a^k, \mu^k\right)$ is a nondegenerate critical point of $\Psi_k$, we see
	\begin{align}\label{equ:241106-e23}
			\operatorname{Rank}\left(D_{(x, \mu) x}^2 \Psi_k(x, \mu)\right)_{(x, \mu)=\left(a^k, \mu^k\right)}=N k .
	\end{align}

	Hence \eqref{equ:241106-e21}, \eqref{equ:241106-e22} and \eqref{equ:241106-e23} imply that
	\begin{align}\label{equ:241106-e24}
			B_{\rho, j, h}=o\big(\frac{1}{\bar{\mu}_{\rho}^{N-1}}\big) \text {, for } j=1, \cdots, k \text { and } h=1, \cdots, N.
	\end{align}

	On the other hand, from \eqref{equ:241015-e12}, \eqref{equ:241106-e4} and \eqref{equ:241016-e1}, we find
	\begin{align} \label{equ:241106-e25}
		B_{\rho, j, h} & =\int_{B_{\mu_{j, \rho}^{(1)} d}(0)} x_h C_{\rho}\Big(\frac{x}{\mu_{j, \rho}^{(1)}}+x_{j, \rho}^{(1)}\Big) \xi_{\rho, j}(x) d x \\ \nonumber
		& =\frac{1}{(\mu_{j, \rho}^{(1)})^{N-1}} \int_{\mathbb{R}^N} x_h U_{0,1}^{\frac{4}{N-2}}\Big(\sum_{l=1}^N c_{j, l} \psi_l(x)\Big) d x+o\big(\frac{1}{\bar{\mu}_{\rho}^{N-1}}\big) \\ \nonumber
		& =-\frac{N-2}{2 N} \int_{\mathbb{R}^N} \frac{|x|^2}{\left(1+|x|^2\right)^{\frac{N}{2}}} \frac{c_{j, h}}{(\mu_{j, \rho}^{(1)})^{N-1}}+o\big(\frac{1}{\bar{\mu}_{\rho}^{N-1}}\big) .
	\end{align}

	Then \eqref{equ:241106-e24} and \eqref{equ:241106-e25} imply \eqref{equ:241106-e26}.
	
\end{proof}

		{\bf{Proof of Theorem \ref{th-4}} }
		For any given $a^k=\left(a_1, \cdots, a_k\right)$, since $M_k\big(a^k\big)$ is a positive matrix and $\left(a^k, \mu^k\right)$ is a nondegenerate critical point of $\Psi_k$, then from\cite{MA} and \cite{CLP}, we find a solution of \eqref{equ:Main-Problem} with \eqref{1.4}. 
	Next, we prove the local uniqueness of solutions to \eqref{equ:Main-Problem} with \eqref{1.4}.
	
	From \eqref{equ:241107-e1}, it holds
	\begin{equation*}
		\left|\xi_{\rho}(x)\right|=O\Big(\frac{1}{R^2}\Big)+O \Big(  \frac{ 1 }{    \bar{\mu}_\rho^{N-2}  } \Big), \text { for } x \in \Omega \backslash \bigcup_{j=1}^k B_{R(\mu_{j, \rho}^{(1)})^{-1}}\big(x_{j, \rho}^{(1)}\big),
	\end{equation*}
	which implies that for any fixed $\gamma \in(0,1)$ and small $\rho$, there exists $R_1>0$,
	\begin{equation}\label{equ:241107-e4}
		\left|\xi_{\rho}(x)\right| \leq \gamma, ~x \in \Omega \backslash \bigcup_{j=1}^k B_{R_1(\mu_{j, \rho}^{(1)})^{-1}}\big(x_{j, \rho}^{(1)}\big) .
	\end{equation}
	Also for the above fixed $R_1$, from \eqref{equ:241107-e2} and \eqref{equ:241106-e26}, we have
	\begin{equation*}
		\xi_{\rho, j}(x)=o(1) ~\text { in } B_{R_1}(0), ~j=1, \cdots, k .
	\end{equation*}
	We know $\xi_{\rho, j}(x)=\xi_{\rho}\Big(\frac{x}{\mu_{j, \rho}^{(1)}}+x_{j, \rho}^{(1)}\Big)$, so
	\begin{align}\label{equ:241107-e3}
			\xi_{\rho}(x)=o(1), x \in \bigcup_{j=1}^k B_{R_1(\mu_{j, \rho}^{(1)})^{-1}}\big(x_{j, \rho}^{(1)}\big) .
	\end{align}

	Hence for any fixed $\gamma \in(0,1)$ and small $\rho$, \eqref{equ:241107-e4} and \eqref{equ:241107-e3} imply $\left|\xi_{\rho}(x)\right| \leq \gamma$ for all $x \in \Omega$, which is in contradiction with $\left\|\xi_{\rho}\right\|_{L^{\infty}(\Omega)}=1$. As a result, $u_{\rho}^{(1)}(x) \equiv u_{\rho}^{(2)}(x)$ for small $\rho$.
	\qed

\section{appendix A. Some basic estimates}	
%
%
%
%
%
%
%
%

The Green's function $G(x, \cdot)$ is the solution of

$$
\begin{cases}-\Delta G(x, \cdot)=\delta_x, & \text { in } \Omega, \\ G(x, \cdot)=0, & \text { on } \partial \Omega,\end{cases}
$$
where $\delta_x$ is the Dirac function. For $G(x, y)$, we have the following form
$$
G(x, y)=S(x, y)-H(x, y),(x, y) \in \Omega \times \Omega,
$$
where $S(x, y)=\frac{1}{(N-2) \omega_N|y-x|^{N-2}}$ is the singular part and $H(x, y)$ is the regular part of $G(x, y), \omega_N$ is a measure of the unit sphere of $\mathbb{R}^N$. For any $x \in \Omega$, we denote $R(x):=H(x, x)$, which is called the Robin function.

	\begin{lemma}
	For any small fixed $d>0$ and $j=1,2 \cdots, k$, it holds
		\begin{align}\label{equ:241007-e1}
		P U_{x_{j,\lambda}, \mu_{j,\lambda}}(x)=O\big(\frac{1}{\mu_{j,\lambda}^{(N-2) / 2}}\big) \text { in } C^1\big(\Omega \backslash B_d(x_{j,\lambda})\big).
	\end{align}

	Moreover, if we define $\varphi_{x_{j,\lambda}, \mu_{j,\lambda}}(x)=U_{x_{j,\lambda}, \mu_{j,\lambda}}(x)-P U_{x_{j,\lambda}, \mu_{j,\lambda}}(x)$, then it holds
	\begin{align}\label{equ:241114-e5}
		\varphi_{x_{j,\lambda}, \mu_{j,\lambda}}(x)=O\big(\frac{1}{\mu_{j,\lambda}^{(N-2) / 2}}\big) \text { in }~ C^1(\Omega).
	\end{align}

\end{lemma} 

\begin{proof}
	See   \cite[ Proposition 1]{Rey}.	
\end{proof}

Similar to \cite{CLP},  we show the local Pohozaev type identities:
\begin{equation}\label{equ:241105-e3}
	-\int_{\partial \Omega^{\prime}} \frac{\partial u_{\rho}}{\partial \nu} \frac{\partial u_{\rho}}{\partial x_i}+\frac{1}{2} \int_{\partial \Omega^{\prime}}\left|\nabla u_{\rho}\right|^2 \nu_i=\frac{N-2}{2 N} \int_{\partial \Omega^{\prime}} u_{\rho}^{\frac{2 N}{N-2}} \nu_i+\frac{\lambda_\rho}{2} \int_{\partial \Omega^{\prime}} u_{\rho}^2 \nu_i,
\end{equation}
and

\begin{equation}\label{equ:241105-e4}
	\begin{aligned}
		& -\int_{\partial \Omega^{\prime}} \frac{\partial u_{\rho}}{\partial \nu}\left\langle x-x_{j, \rho}, \nabla u_{\rho}\right\rangle+\frac{1}{2} \int_{\partial \Omega^{\prime}}\left|\nabla u_{\rho}\right|^2\left\langle x-x_{j, \rho}, \nu\right\rangle+\frac{2-N}{2} \int_{\partial \Omega^{\prime}} \frac{\partial u_{\rho}}{\partial \nu} u_{\rho} \\
		& \quad=\frac{N-2}{2 N} \int_{\partial \Omega^{\prime}} u_{\rho}^{\frac{2 N}{N-2}}\left\langle x-x_{j, \rho}, \nu\right\rangle+\frac{\lambda_\rho}{2} \int_{\partial \Omega^{\prime}} u_{\rho}^2\left\langle x-x_{j, \rho}, \nu\right\rangle-\lambda_\rho \int_{\Omega^{\prime}} u_{\rho}^2.
	\end{aligned}
\end{equation}

\vspace{1cm}
\noindent\textbf{Acknowledgments} Xiaoyu Zeng is supported by NSFC (Grant Nos. 12322106, 12171379, 12271417) . Huan-Song Zhou  is supported by NSFC (Grant Nos. 11931012, 12371118) .

\vskip.2truein  
\noindent\textbf{Data availability} Data sharing is not applicable to this article as no datasets were generated or analyzed during
the current study.

\vskip.2truein  
\noindent\textbf{Declarations} Conflict of interest The authors declare that there is no conflict of interest.

\end{document}